\documentclass[11pt]{article}

\usepackage[english]{babel}

\usepackage[letterpaper,top=2cm,bottom=2cm,left=3cm,right=3cm,marginparwidth=1.75cm]{geometry}

\usepackage{amsbsy,amssymb,amsmath,amsthm,amscd,amsfonts,latexsym}
\usepackage{graphicx}
\usepackage{microtype}
\usepackage[colorlinks=true, allcolors=blue]{hyperref}
\usepackage{indentfirst}
\usepackage{tikz-cd}
\newtheorem{thm}{Theorem}
\newtheorem{prop}{Proposition}
\newtheorem{lem}[prop]{Lemma}
\newtheorem{cor}[prop]{Corollary}

\newtheorem{conj}[prop]{Conjecture}

\theoremstyle{definition}
\newtheorem{df}[prop]{Definition}

\newtheorem{rmk}[prop]{Remark} 
\newtheorem{ex}[prop]{Example}
\newtheorem{qtn}[prop]{Question}

\newcommand{\cM}{\mathcal{M}}
\newcommand{\cJ}{\mathcal{J}}
\newcommand{\cH}{\mathcal{H}}
\newcommand{\cB}{\mathcal{B}}

\newcommand{\R}{{\mathbb{R}}}
\newcommand{\Z}{{\mathbb{Z}}}
\newcommand{\C}{{\mathbb{C}}}

\newcommand{\N}{{\mathbb{N}}}

\newcommand{\cP}{\mathcal{P}}

\newcommand{\bK}{{\mathbb{K}}}

\newcommand{\cO}{{\mathcal{O}}}
\newcommand{\ahl}{\mathcal{A}_{L_0,H}}
\newcommand{\adH}{\mathcal{H}_{\rm ad}}
\newcommand{\bh}{\hbar_{\rm bar}}
\newcommand{\sbh}{\hbar_{\rm \mathcal{S}bar}}
\newcommand{\ph}{\hbar_{\rm per}}
\newcommand{\sph}{\hbar_{\rm \mathcal{S}per}}

\newcommand{\ham}{\operatorname{Ham} }
\newcommand{\symp}{\operatorname{Symp}}

\newcommand{\im}{\operatorname{im}}

\counterwithin{equation}{section}
\counterwithin{thm}{section}
\counterwithin{prop}{section}

\title{Persistent Entropy of Floer Persistence Barcodes}
\author{Wenmin Gong}

\begin{document}
\maketitle

\begin{abstract}
Floer persistence barcodes provide a quantitative way to encode action-filtered Floer homology. Inspired by the Shannon entropy of persistence barcodes in topological data analysis, we introduce a Floer-theoretic entropy invariant, called \textit{persistent entropy}, which measures the asymptotic linear growth rate, under iteration, of the Shannon entropy determined by the distribution of finite bar lengths. This is complementary to the barcode entropy of \c{C}ineli--Ginzburg--G\"{u}rel, which records the exponential growth rate of the number of  not-too-short bars. We prove that, for Hamiltonian diffeomorphisms, the relative and absolute persistent entropies coincide with the corresponding barcode entropies. For Liouville domains, we prove general comparison inequalities and a subexponential length-growth criterion which gives equality beyond the case of vanishing symplectic homology. We also compute the persistent entropy of cotangent disk bundles of negatively curved manifolds and relate it to the topological entropy of the geodesic flow. In addition, we prove Hofer-stability estimates for finite-level Shannon entropy and derive flexibility and rigidity-type questions for barcode and persistent entropies of Reeb flows.
\end{abstract}

\section{Introduction}\label{sec:1}

Persistent homology has become a useful bridge between symplectic topology and dynamical systems. In Floer theory, the action filtration turns a Floer complex into a persistence module, and the associated barcode records both algebraic information and quantitative action data. Following the work of Polterovich and Shelukhin~\cite{PS}, Floer persistence has been used in many different directions in symplectic topology; see, for instance,~\cite{BP3S2,LNV,LSV,She,UZ}. A particularly influential invariant in this circle of ideas is the \textit{barcode entropy} introduced by \c{C}ineli--Ginzburg--G\"{u}rel~\cite{CGG}. Roughly speaking, barcode entropy measures the exponential growth rate, under iteration, of the number of bars in the Floer barcode whose lengths are larger than a prescribed positive number. It is closely related to topological entropy: in~\cite{CGG}, barcode entropy is bounded above by topological entropy, while compact hyperbolic invariant sets give lower bounds for absolute barcode entropy. In particular, on closed surfaces the absolute barcode entropy of a Hamiltonian diffeomorphism agrees with its topological entropy. In higher dimensions, however, the relation is subtler; for example, \c{C}ineli~\cite{Cin} constructed Hamiltonian pseudo-rotations with positive topological entropy but zero absolute barcode entropy.

Barcode entropy counts long bars, but it does not use the distribution of their lengths. This distribution is itself meaningful. The length of the longest finite bar is the boundary depth~\cite{Us2,Us3}, an invariant with important applications to Hofer geometry and to partial results on the Hofer diameter and Hofer--Zehnder conjectures~\cite{Us2,She}. Thus it is natural to ask whether one can extract a dynamical entropy invariant not only from the number of bars, but also from how the total finite persistence is distributed among them. The starting point of this paper is to apply Shannon entropy to the normalized vector of finite bar lengths. Similar Shannon-type quantities are standard in topological data analysis, where they are often called persistent entropy and enjoy stability under perturbations of persistence diagrams~\cite{AGR,AGS,CGGJK,RCMP}. In the present paper, however, the term \textit{persistent entropy} refers to the growth rate of these Shannon entropies under iteration of Hamiltonian diffeomorphisms or along the action window for Reeb flows.

The guiding question is whether this length-sensitive invariant detects asymptotic information beyond barcode entropy. Our main results show that, for Hamiltonian diffeomorphisms, this does not occur at the level of asymptotic growth rates: the relative and absolute persistent entropies, defined as the linear growth rates of the Shannon entropies of the normalized finite-bar length distributions, agree with the corresponding barcode entropies. Equivalently, the effective number of finite bars encoded by the length distribution has the same exponential growth rate as the number of finite bars whose lengths are bounded from below. Thus, in the Hamiltonian setting, the bar-length distribution and the count of not-too-short bars carry the same asymptotic information at the exponential scale. In the contact setting the situation is more delicate.
 We prove the general inequalities
\[
\ph(X,\lambda)\leq \bh(X,\lambda),\qquad
\sph(X,\lambda)\leq \sbh(X,\lambda),
\]
and obtain equality when the symplectic homology of the filling vanishes. More generally, equality follows whenever the maximal length of the bars counted in the action window grows subexponentially; see Proposition~\ref{prop:subexp-length}. This isolates the only possible mechanism for a strict gap: finite bars whose lengths grow exponentially compared with the action window. A purely algebraic example, Example~\ref{ex:algebraic-gap}, shows that such a gap is possible for abstract barcode families, even though its symplectic realizability remains open. Thus, the equality results in the Hamiltonian and controlled Liouville settings should be viewed as structural Floer-theoretic statements rather than tautological consequences of the definitions.

Another feature of the Shannon-entropy viewpoint is stability. While the asymptotic equalities with barcode entropy are the main focus of the paper, the finite-level Shannon entropy retains information that is not visible from bar counts alone. Using the stability of persistence barcodes and the Hofer-continuity of Floer barcodes, we prove quantitative Hofer-stability estimates for the Shannon entropy of Floer barcodes; see Theorems~\ref{thm:E-stab} and~\ref{thm:E-stab'}. These estimates show that persistent entropy is not merely a formal rephrasing of barcode entropy: before passing to the asymptotic limit, it is a stable numerical statistic of the Floer barcode which is sensitive to the distribution of bar lengths. This stability is one of the motivations for considering Shannon entropy in Floer persistence.

\subsection{Definitions of relative and absolute persistent entropies for Hamiltonian diffeomorphisms}\label{subsec:def-perEntropy}

We now give the definitions used in the paper. They are modeled on the absolute and relative barcode entropies of \c{C}ineli--Ginzburg--G\"{u}rel~\cite{CGG}. Let $(M,\omega)$ be a symplectic manifold which is either compact or sufficiently well behaved at infinity, for instance convex at infinity. Let $L_0,L_1\subset M$ be closed monotone Hamiltonian isotopic Lagrangian submanifolds with minimal Maslov number $N_{L_0}\geq 2$. We refer to Section~\ref{subsec:notation} for the precise conventions.

Assume first that $L_1=\varphi_H^{-1}(L_0)$ for some Hamiltonian $H\in C_c^\infty([0,1]\times M,\R)$ and that $L_0$ and $L_1$ intersect transversely. The Lagrangian Floer complex $CF(L_0,H)$ is generated by Hamiltonian $1$-chords with endpoints on $L_0$ in all path components, and it carries an action filtration. We assume throughout that the corresponding Floer homology, denoted by $HF(L_0,H)$, is nonzero. The intersection points $x\in L_0\cap L_1$ are in one-to-one correspondence with Hamiltonian $1$-chords $t\mapsto \varphi_H^t(x)$. Let $\mathcal B(L_0,H)$ be the barcode of this filtered complex; see Section~\ref{sec:floer}. For $\epsilon>0$, let $\mathcal{FB}^\epsilon(L_0,H)$ be the multiset of finite bars in $\mathcal B(L_0,H)$ whose lengths are larger than $\epsilon$. We write
\[
\mathcal{FB}^\epsilon(L_0,H)
=\big\{([a_i],l_i)\in\R/\Gamma\times [0,\infty]
\mid \epsilon<l_i<\infty,
1\leq i\leq k\big\},
\]
where $\Gamma\subset\R$ is the period group associated with $\omega$; see Section~\ref{subsec:lfc}. Set $S_k=\sum_{i=1}^k l_i$. The Shannon entropy of $\mathcal{FB}^\epsilon(L_0,H)$ is
\[
E^\epsilon(L_0,L_1)
=
-\sum_{i=1}^k \frac{l_i}{S_k}\log\frac{l_i}{S_k},
\]
with the convention that $E^\epsilon(L_0,L_1)=0$ if $\mathcal{FB}^\epsilon(L_0,H)$ is empty. Throughout the paper, logarithms are taken in base $2$, and we set $p\log p=0$ for $p=0$.

When $L_0$ and $L_1$ are not transverse, we define
\[
E^\epsilon(L_0,L_1)=\liminf_{L'\to L_1}E^\epsilon(L_0,L'),
\]
where the limit is taken over smooth Lagrangian submanifolds $L'$ which are Hamiltonian isotopic to $L_1$, intersect $L_0$ transversely, and converge to $L_1$ in the $C^\infty$ topology.

\begin{df}[Relative Persistent Entropy I]\label{df:relper1.1}
Let $\varphi\in\ham(M,\omega)$. The \textit{$\epsilon$-persistent entropy of $\varphi$ relative to $(L_0,L_1)$} is
\[
\ph^\epsilon(\varphi;L_0,L_1)
=
\limsup_{n\to\infty}
\frac{E^\epsilon(L_0,\varphi^{-n}(L_1))}{n},
\]
and the \textit{persistent entropy of $\varphi$ relative to $(L_0,L_1)$} is
\[
\ph(\varphi;L_0,L_1)
=
\limsup_{\epsilon\searrow0}\ph^\epsilon(\varphi;L_0,L_1)
\in[0,\infty].
\]
\end{df}

If $L\subset M$ is a closed monotone Lagrangian submanifold with $N_L\geq2$, we also use the following single-Lagrangian version.

\begin{df}[Relative Persistent Entropy II]\label{df:relperII}
The \textit{$\epsilon$-persistent entropy of $\varphi$ relative to $L$} is
\[
\ph^\epsilon(\varphi;L)
=
\limsup_{n\to\infty}
\frac{E^\epsilon(L,\varphi^{-n}(L))}{n},
\]
and the \textit{persistent entropy of $\varphi$ relative to $L$} is
\[
\ph(\varphi;L)
=
\limsup_{\epsilon\searrow0}\ph^\epsilon(\varphi;L)
\in[0,\infty].
\]
\end{df}

We next define the absolute version. Let $(M,\omega)$ be a closed monotone symplectic manifold and let $\varphi=\varphi_H^1$ be a Hamiltonian diffeomorphism. One can either apply the relative construction to the diagonal $\Delta\subset (M\times M,-\omega\oplus\omega)$ and the map $id\times\varphi$, or work directly with the Hamiltonian Floer complex $CF_*(H)$ generated by Hamiltonian $1$-periodic orbits over $\Lambda$ in all free homotopy classes. Applying singular value decomposition to this filtered complex gives a barcode $\mathcal B(H)$. If $G$ is another Hamiltonian with $\varphi_G^1=\varphi_H^1$, then the barcodes in each free homotopy class differ only by a filtration shift; see Remark~\ref{rmk:shift}. In particular, the finite bar lengths, and hence their Shannon entropy, are independent of the chosen Hamiltonian up to the standard limiting procedure.

Let $E^\epsilon(\varphi^n)$ be the lower limit of the Shannon entropies of $\mathcal{FB}^\epsilon(H')$ as $H'$ converges to $H^{\sharp n}$ in the $C^\infty$ topology and $(id\times\varphi_{H'}^1)(\Delta)$ is transverse to $\Delta$. Then
\begin{equation}\label{e:H=L}
E^\epsilon(\varphi^n)=E^\epsilon(\Delta,(id\times\varphi)^{-n}(\Delta)).
\end{equation}

\begin{df}[Absolute Persistent Entropy]\label{df:relperIII}
The \textit{$\epsilon$-persistent entropy of $\varphi$} is
\[
\ph^\epsilon(\varphi)
=
\limsup_{n\to\infty}\frac{E^\epsilon(\varphi^n)}{n},
\]
and the \textit{absolute persistent entropy of $\varphi$} is
\[
\ph(\varphi)
=
\limsup_{\epsilon\searrow0}\ph^\epsilon(\varphi)
\in[0,\infty].
\]
\end{df}

For zero-entropy systems, it is also useful to consider polynomial growth rates, as in slow entropy theory~\cite{KT,FS,BG}.

\begin{df}[Slow Persistent Entropy]
The \textit{relative slow persistent entropy of $\varphi$ with respect to $(L_0,L_1)$} is
\[
\sph(\varphi;L_0,L_1)
=
\limsup_{\epsilon\searrow0}\sph^\epsilon(\varphi;L_0,L_1)
\in[0,\infty],
\]
where
\[
\sph^\epsilon(\varphi;L_0,L_1)
=
\limsup_{n\to\infty}
\frac{E^\epsilon(L_0,\varphi^{-n}(L_1))}{\log n}.
\]
The absolute slow persistent entropy $\sph(\varphi)$ is defined similarly by replacing $E^\epsilon(L_0,\varphi^{-n}(L_1))$ with $E^\epsilon(\varphi^n)$.
\end{df}

\begin{rmk}[Exact Lagrangians]
Definition~\ref{df:relper1.1} and the construction extend word-for-word to exact Lagrangian submanifolds.
\end{rmk}

\subsection{Persistent entropy of Liouville domains}\label{subsec:def-per-Liou}

Let $(X,\lambda)$ be a Liouville domain, i.e. a compact exact symplectic manifold with boundary such that the Liouville vector field $Z$, defined by $d\lambda(Z,-)=\lambda$, points outward along $\partial X$. Then $(\partial X,\theta:=\lambda|_{\partial X})$ is a contact manifold with contact structure $\xi=\ker\theta$. A contact manifold $(Y,\xi)$ is said to be Liouville domain fillable if $Y=\partial X$ and $\xi=\ker(\lambda|_Y)$ for some Liouville domain $(X,\lambda)$.

Fix a ground field $\bK$. We denote by $SH^s(X,\lambda)$ the filtered symplectic homology of $(X,\lambda)$ over $\bK$ below action level $s$. The family $\{SH^s(X,\lambda)\}_{s\in\R}$ is a persistence module, and we write $\mathcal B(X,\lambda)$ for its barcode; see~\cite{CGG2,FLS,PRSZ} and Section~\ref{subsec:symphomology}.

For $\epsilon>0$ and $s\in\R$, let $E_s^\epsilon(X,\lambda)$ be the Shannon entropy of the multiset $\mathcal{FB}_{s}^\epsilon(X,\lambda)$ consisting of the following finite intervals: finite bars $(a,b]$ in $\mathcal B(X,\lambda)$ with $b-a>\epsilon$ and $a<s-\epsilon$, and truncated infinite bars $(a,s]$ coming from bars $(a,\infty)$ with $a<s-\epsilon$. Equivalently,
\[
E_s^\epsilon(X,\lambda)
=E(\mathcal{FB}_{s}^\epsilon(X,\lambda)),
\]
where $E$ denotes the Shannon entropy of a barcode; see Definition~\ref{df:Shanentropy}.

\begin{df}\label{df:per_Liou}
The \textit{$\epsilon$-persistent entropy} of $(X,\lambda)$ is
\[
\ph^\epsilon(X,\lambda)
=
\limsup_{s\to\infty}\frac{E_s^\epsilon(X,\lambda)}{s},
\]
and the \textit{persistent entropy} of $(X,\lambda)$ is
\[
\ph(X,\lambda)
=
\limsup_{\epsilon\searrow0}\ph^\epsilon(X,\lambda).
\]
The \textit{slow persistent entropy} of $(X,\lambda)$ is
\[
\sph(X,\lambda)
=
\limsup_{\epsilon\searrow0}
\limsup_{s\to\infty}
\frac{E_s^\epsilon(X,\lambda)}{\log s}.
\]
\end{df}

\subsection{Main results and open problems}

Recall that $\bh(\varphi;L_0,L_1)$ and $\sbh(\varphi;L_0,L_1)$ denote the relative barcode entropy and slow barcode entropy of $\varphi$ with respect to $(L_0,L_1)$, respectively. Our first main result says that, the linear growth rate of the Shannon entropy of finite bar lengths coincides the exponential growth rate of the number of not-too-short bars.

\begin{thm}\label{thm:Lpervsbar}
Let $L_0$ and $L_1$ be two closed Hamiltonian isotopic monotone Lagrangian submanifolds of a tame symplectic manifold $(M,\omega)$ with $N_{L_0}\geq 2$. Then, for any $\varphi\in\ham(M,\omega)$, one has
\begin{equation}\label{e:perbar}
\ph(\varphi;L_0,L_1)= \bh(\varphi;L_0,L_1).
\end{equation}
Moreover,
\begin{equation}\label{e:sperbar}
\sph(\varphi;L_0,L_1)\leq \sbh(\varphi;L_0,L_1).
\end{equation}
If the boundary depths $\beta_n:=\beta(L_0,\varphi^{-n}(L_1))$ satisfy
\[
\limsup_{n\to\infty}\frac{\log \beta_n}{\log n}=0,
\]
then equality holds in~\eqref{e:sperbar}.
\end{thm}

\begin{cor}\label{cor:Lpervsbar}
Let $L$ be a closed monotone Lagrangian submanifold with minimal Maslov number $N_L\geq2$ in a tame symplectic manifold $(M,\omega)$. Then, for every $\varphi\in\ham(M,\omega)$,
\[
\ph(\varphi;L)=\bh(\varphi;L).
\]
\end{cor}

Since $\ph(id\times\varphi;\Delta)=\ph(\varphi)$ and $\bh(id\times\varphi;\Delta)=\bh(\varphi)$, the relative theorem gives the absolute statement.

\begin{thm}\label{thm:abpervsbar}
Let $\varphi:M\to M$ be a Hamiltonian diffeomorphism of a closed monotone symplectic manifold $(M,\omega)$. Then
\[
\ph(\varphi)=\bh(\varphi),
\]
and
\begin{equation}\label{e:abslowper}
\sph(\varphi)\leq\sbh(\varphi).
\end{equation}
Moreover, if the boundary depths $\beta_n:=\beta(\varphi^n)$ satisfy
\[
\limsup_{n\to\infty}\frac{\log \beta_n}{\log n}=0,
\]
then equality holds in~\eqref{e:abslowper}.
\end{thm}

Combining Theorems~\ref{thm:Lpervsbar} and~\ref{thm:abpervsbar} with the upper bounds for barcode entropy from~\cite{CGG}, we obtain finiteness of the corresponding persistent entropies for compactly supported Hamiltonian diffeomorphisms. Moreover, recent work of \c{C}ineli--Ginzburg--G\"{u}rel~\cite{CGG4} relates barcode entropy to metric entropy of pseudo-chord measures. Through the equalities above, persistent entropy inherits the same dynamical lower-bound properties.

The stability estimates proved later in the paper complement these asymptotic identities. At the finite-barcode level, Theorems~\ref{thm:E-stab} and~\ref{thm:E-stab'} give explicit continuity estimates for Shannon entropy with respect to Hofer distance, for Lagrangian pairs and Hamiltonian diffeomorphisms respectively. These estimates are not needed for the proof of Theorems~\ref{thm:Lpervsbar} and~\ref{thm:abpervsbar}; rather, they show that the length-distribution statistic underlying persistent entropy is compatible with the natural metric geometry of Hamiltonian dynamics.

Assume now that $(X,\lambda)$ is a Liouville domain. Let $\bh(X,\lambda)$ and $\sbh(X,\lambda)$ denote the barcode entropy and slow barcode entropy of the Reeb flow of the contact form $\theta=\lambda|_{\partial X}$; see Definition~\ref{df:bar_Liou}. The contact analogue of the Hamiltonian comparison theorem is the following.

\begin{thm}\label{thm:liouville}
Let $(X,\lambda)$ be a Liouville domain. Then
\[
\ph(X,\lambda)\leq\bh(X,\lambda),
\qquad
\sph(X,\lambda)\leq\sbh(X,\lambda).
\]
Moreover, if $SH(X,\lambda)=0$, then both inequalities are equalities.
\end{thm}

The proof gives a more flexible criterion: the vanishing of $SH(X,\lambda)$ can be replaced by a subexponential upper bound on the maximal length of the bars counted in the action window; see Proposition~\ref{prop:subexp-length}. This criterion will also be used below for cotangent disk bundles of negatively curved manifolds, where the symplectic homology is typically nonzero.

\begin{rmk}
A direct consequence of Theorem~\ref{thm:liouville} is that, for any star-shaped domain $W\subset\R^{2n}$ with smooth boundary, the barcode entropy and persistent entropy of the Reeb flow on $\partial W$ agree; the same is true for their slow versions.
\end{rmk}

The equality part of Theorem~\ref{thm:liouville} uses the vanishing of symplectic homology. This removes infinite bars from the full symplectic homology barcode and provides uniform control on the finite bars. It is natural to ask whether this assumption is essential.

\begin{qtn}\label{qtn:liouville-equality}
Let $(X,\lambda)$ be a Liouville domain. Is it always true that
\[
\ph(X,\lambda)=\bh(X,\lambda),
\qquad
\sph(X,\lambda)=\sbh(X,\lambda)?
\]
Equivalently, can the assumption $SH(X,\lambda)=0$ in Theorem~\ref{thm:liouville} be removed? If not, can one construct a Liouville domain with $SH(X,\lambda)\neq0$ for which one of the inequalities
\[
\ph(X,\lambda)\leq\bh(X,\lambda),
\qquad
\sph(X,\lambda)\leq\sbh(X,\lambda)
\]
is strict? More generally, can the possible gap be detected by the infinite bars or by the asymptotic distribution of long finite bars in the symplectic homology barcode?
\end{qtn}

Persistent entropy is positive whenever the corresponding barcode entropy is forced to be positive. We record two consequences of known lower bounds for barcode entropy. First, if $K\subset M$ is a compact hyperbolic invariant set of a Hamiltonian diffeomorphism, then the lower bounds of \c{C}ineli--Ginzburg--G\"{u}rel~\cite{CGG} and Meiwes~\cite{Me}, together with Theorems~\ref{thm:Lpervsbar} and~\ref{thm:abpervsbar}, give the following.

\begin{thm}\label{thm:lbd}
Let $\varphi:M\to M$ be a Hamiltonian diffeomorphism of a closed monotone symplectic manifold $(M,\omega)$, and let $K\subset M$ be a compact hyperbolic invariant set of $\varphi$. If $h_{\mathrm{top}}(\varphi|_K)>0$, then
\[
0<h_{\mathrm{top}}(\varphi|_K)\leq\ph(\varphi).
\]
\end{thm}

\begin{thm}
Let $L_0$ and $L_1$ be two closed Hamiltonian isotopic monotone Lagrangian submanifolds of a tame symplectic manifold $(M,\omega)$ with $N_{L_0}\geq2$. Let $K$ be a locally maximal and topologically transitive hyperbolic invariant set of $\varphi\in\ham(M,\omega)$ with $h_{\mathrm{top}}(\varphi|_K)>0$. Assume that $L_0$ and $L_1$ intersect $K$, that $L_0$ contains a ball centered at some $p\in L_0\cap K$ in the unstable manifold of $K$ at $p$, and that $L_1$ contains a ball centered at some $q\in L_1\cap K$ in the stable manifold of $K$ at $q$. Then
\[
0<h_{\mathrm{top}}(\varphi|_K)\leq\ph(\varphi;L_0,L_1).
\]
\end{thm}

For slow entropy, the boundary-depth condition in Theorems~\ref{thm:Lpervsbar} and~\ref{thm:abpervsbar} is automatic in several important cases. For instance, Shelukhin~\cite[Thm.~B]{Sh} proved uniform bounds for boundary depth under suitable semisimplicity and toroidal monotonicity assumptions. Consequently, we obtain the following corollary.

\begin{cor}
Assume that $(M^{2n},\omega)$ is a closed simply connected monotone symplectic manifold whose even quantum homology algebra over a ground field $\bK$ is semisimple. Then, for every $\varphi\in\ham(M,\omega)$,
\[
\sph(\varphi)=\sbh(\varphi).
\]
\end{cor}

In particular, for complex projective space $\C P^n$ with the Fubini--Study form $\omega_{\rm FS}$, one has $\sph(\varphi)=\sbh(\varphi)$ for every $\varphi\in\ham(\C P^n,\omega_{\rm FS})$.

We also obtain positivity results for Reeb flows. Fender--Lee--Sohn~\cite{FLS} proved that the barcode entropy of a Reeb flow is bounded above by topological entropy, namely
\begin{equation}\label{eq:bartop}
 h_{\mathrm{top}}(\theta)\geq\bh(X,\lambda).
\end{equation}
Conversely, compact hyperbolic invariant sets give lower bounds for barcode entropy by work of \c{C}ineli--Ginzburg--G\"{u}rel~\cite{CGG2}. Combining these results with Theorem~\ref{thm:liouville}, we get the following.

\begin{thm}\label{thm:positivity}
Let $(X,\lambda)$ be a Liouville domain with $SH(X,\lambda)=0$. Let $K\subset\partial X$ be a compact hyperbolic invariant set of the Reeb flow $\varphi_\theta^t$. Then
\[
\ph(X,\lambda)\geq h_{\mathrm{top}}(\varphi_\theta^t|_K).
\]
Moreover, when $\dim X=3$,
\[
\ph(X,\lambda)=h_{\mathrm{top}}(\theta).
\]
\end{thm}

In particular, if $W\subset\R^{2n}$ is a star-shaped domain with smooth boundary and the Reeb flow on $\partial W$ has a compact hyperbolic invariant set, then $\ph(W,\lambda)>0$.

Finally, we discuss flexibility. Katok's entropy rigidity theorem~\cite{Ka} for geodesic flows on surfaces states that if $S$ is a closed surface with negative Euler characteristic, then every Riemannian metric $g$ on $S$ satisfies
\[
h_{\mathrm{top}}(g)
\geq
\sqrt{-2\pi\chi(S)/\mathrm{area}_g(S)},
\]
with equality only in constant curvature. This result motivates the collapse and rigidity questions for topological entropy of Reeb flows studied by Abbondandolo--Alves--Sa\u{g}lam--Schlenk~\cite{AASS}. For barcode and persistent entropies we prove the following flexibility statements.

For a compact symplectic manifold $(M^{2n},\omega)$, set
\[
\mathrm{vol}_\omega(M)=\frac{1}{n!\omega_n}\int_M\omega^n,
\]
where $\omega_n$ is the Euclidean volume of the unit ball in $\R^n$.

\begin{thm}\label{thm:collapse}
Let $(Y,\xi)$ be a contact manifold fillable by a Liouville domain $X^{2n}$. For every $\varepsilon>0$, there exists a Liouville form $\lambda$ on $X$ such that $\xi=\ker(\lambda|_Y)$, $\mathrm{vol}_{d\lambda}(X)=1$, and
\[
\ph(X,\lambda)\leq\bh(X,\lambda)\leq\varepsilon.
\]
\end{thm}

The next theorem is an immediate consequence of Theorem~A in~\cite{EK}; see also~\cite[Thm.~1.9]{AASS} for a related higher-dimensional result on topological entropy of Reeb flows.

\begin{thm}\label{thm:flexibility}
Let $S$ be a closed orientable surface of genus $k\geq2$. Then for every $c>2\sqrt{\pi(k-1)}$ 
there exists a Riemannian metric $g$ on $S$ such that the canonical Liouville form $\lambda$ on the disk cotangent bundle $D_g^*S$ satisfies $\mathrm{vol}_{d\lambda}(D_g^*S)=1$ and
\[
\ph(D_g^*S,\lambda)=\bh(D_g^*S,\lambda)=c.
\]
\end{thm}

Theorem~\ref{thm:flexibility} realizes all values strictly larger than Katok's constant $2\sqrt{\pi(k-1)}$ . The endpoint itself is not mysterious for barcode entropy: using the barcode entropy of geodesic-flow theorem of Ginzburg--G\"{u}rel--Mazzucchelli~\cite[Thm.C]{GGM}, together with Katok's entropy rigidity theorem~\cite{Ka}, one obtains the following endpoint statement.

\begin{cor}\label{cor:surface-endpoint}
Let $S$ be a closed orientable surface of genus $k\geq2$, and let $g$ be a Riemannian metric on $S$ such that
$\mathrm{vol}_{d\lambda}(D_g^*S)=1$.
Then
\[
\bh(D_g^*S,\lambda_{\mathrm{can}})\geq 2\sqrt{\pi(k-1)}.
\]
Moreover, equality holds if and only if $g$ has constant negative curvature. For such a constant curvature metric,
$\ph(D_g^*S\lambda_{\mathrm{can}})=\bh(D_g^*S,\lambda_{\mathrm{can}})=2\sqrt{\pi(k-1)}.$

\end{cor}

This leaves a sharper question which is genuinely about the length distribution measured by persistent entropy, rather than about bar-count growth.

\begin{qtn}\label{qtn:persistent-endpoint-rigidity}
Let $S$ be a closed orientable surface of genus $k\geq2$, and let $g$ be a Riemannian metric on $S$ with
\[
\mathrm{vol}_{d\lambda}(D_g^*S)=1.
\]
Is it true that
\[
\ph(D_g^*S,\lambda_{\mathrm{can}})\geq 2\sqrt{\pi(k-1)}?
\]
If equality holds, must $g$ have constant negative curvature?
\end{qtn}

There is also a natural higher-dimensional analogue of
Theorem~\ref{thm:flexibility}. By Proposition~\ref{prop:negcurv-cotangent},
for a closed negatively curved Riemannian manifold the persistent and
barcode entropies of the cotangent disk bundle agree with the
topological entropy of the geodesic flow. since negatively curved metrics have no conjugate points, the topological entropy of the geodesic flow agrees with the volume entropy by Manning's theorem~\cite{Ma2}; see also Freire--Ma\~{n}\'{e}~\cite{FM}. Thus the endpoint in the higher-dimensional problem is governed
by the minimal entropy theorem of Besson--Courtois--Gallot~\cite{BCG}.
The remaining issue is a Riemannian entropy flexibility problem above
this BCG threshold.

\begin{conj}[Flexibility above the Besson--Courtois--Gallot threshold]\label{conj:high-dimensional-flexibility}
Let \(M^n\) be a closed manifold admitting a negatively curved locally symmetric metric \(g_0\). Define
\[
C(M)
:=
h_{\mathrm{top}}(\varphi_{g_0})\,
\mathrm{vol}_{d\lambda}(D_{g_0}^*M)^{1/n}.
\]
Then, for every
\[
c>C(M),
\]
there exists a negatively curved Riemannian metric \(g\) on \(M\) such that
\[
\mathrm{vol}_{d\lambda}(D_g^*M)=1
\]
and
\[
\ph(D_g^*M,\lambda_{\mathrm{can}})
=
\bh(D_g^*M,\lambda_{\mathrm{can}})
=
c.
\]
\end{conj}

\section*{Acknowledgements} I would like to thank Jinxin Xue for calling my attention to the papers~\cite{AGR,AGS} and for his grant support NSFC 12271285 and the New Cornerstone investigator program. I am very grateful to Matthias Meiwes for carefully reading the first draft of this manuscript and providing valuable comments. I thank Dylan Cant, Erman Çineli and Viktor L. Ginzburg for illuminating discussions. I thank Wentian Kuang, the organizer of the 2026 GBU Symposium on Celestial Mechanics and Hamiltonian Systems in Dongguan (June 2026) for giving me an opportunity to
present a preliminary version of this work.

\section{Preliminaries}\label{sec:pre}

\subsection{Some conventions}\label{subsec:notation}
Throughout this paper, we assume that $(M,\omega)$ is a connected and tame symplectic
manifold, e.g., closed symplectic manifolds, open manifolds
that are symplectically convex at infinity;  see~\cite{ALP}. We denote
by $\cJ$ the space of $\omega$-compatible almost complex structures such that
$(M,g_{J})$ is geometrically bounded, where for $J\in\cJ$, $g_{J}
(\cdot,\cdot)=\omega(\cdot,J\cdot)$ is the associated Riemannian metric. Let $\cH= C_c^\infty([0,1]\times M,\R)$. For $H\in\cH$,  we denote by
$\{\varphi_H^t\}_{t\in[0,1]}$ the Hamiltonian isotopy of $H$ obtained by
integrating the time-dependent vector field $X_{H_t}$, where $H_t=H(t,\cdot)$, and
$X_{H_t}$ is determined uniquely by $-dH_t=\omega(X_{H_t},\cdot)$. Let $\ham(M,\omega)$ denote the group of all Hamiltonian diffeomorphisms $\varphi$ generated by elements $H$ of $\cH$, i.e.,  $\varphi=\varphi^1_H$. Note that the Hamiltonian $\overline{H}=-H_t\circ\varphi_H^t$ generates $\varphi^{-1}$. We say that $\varphi\in\ham(M,\omega)$ is \textit{nondegenerate} if for each $x\in {\rm Fix}(\varphi)$ the linearized map $d\varphi_x:T_xM\to T_xM$ has all eigenvalues distinct from $1$.

For any $H\in\cH$, we define
\[
\|H\|=\int^1_0(\max_M H_t-\min_M H_t) dt.
\]
The \textit{Hofer norm} of $\varphi\in\ham(M,\omega)$ is defined by
\[
\|\varphi\|_H=\inf_{G\in\cH} \{\|G\||\;\varphi=\varphi^1_G\};
\]
see~\cite{Ho,HZ,Po}. This induces the $\ham$-bi-invariant \textit{Hofer distance} 
\[
d_H(\varphi,\psi):=\|\varphi^{-1}\circ\psi\|_H\quad 
\]
for any $\varphi,\psi\in\ham(M,\omega)$. The Hofer distance between two Hamiltonian isotopic
Lagrangian submanifolds $L_1$ and $L_2$ is defined as
\[
\delta_H(L_1,L_2)=\inf\big\{\|\varphi\|_H\big|\;\varphi(L_1)=L_2,\;\varphi\in \ham(M,\omega)\big\};
\]
see~\cite{Ch,KS}.

We say that a Lagrangian submanifold $L\subseteq M$ is \textit{monotone} if the two homomorphisms
$$\omega:\pi_2(M,L)\to\R,\quad \mu: \pi_2(M,L)\to\Z,$$
given by pairing with the symplectic form $\omega$ and the Maslov class $\mu\in H^2(M,L;\Z)$ satisfy
$$\omega=\kappa_L\mu$$
for some positive constant $\kappa_L$.  In this case it is known that
$(M,\omega)$ is spherically \textit{monotone}, which means that
$$\omega(A)=2\kappa_Lc_1(A),\quad \forall A\in\pi_2(M),$$
where $c_1=c_1(TM,\omega)$ is the first Chern class of $M$. 

The \textit{minimal Maslov number} of $L$ is defined as the integer
$$N_L=\min\big\{\mu(\beta)\;|\;\beta\in C^0(S^1\times[0,1],M)\;\hbox{with}\;\beta(S^1\times\{0,1\})\subset L,\;\mu(\beta)>0\big\}.$$ 
Throughout this paper, we assume that all $L$ are connected and closed monotone submanifolds with $N_L\geq 2$.  

\subsection{The Lagrangian Floer complex}\label{subsec:lfc}
In this subsection we take a Lagrangian submanifold $L_0$ as in Section~\ref{subsec:notation}, and let $L_1=\varphi_H^{-1}(L_0)$ for some fixed  $H\in\cH$. Assume that $L_0$ and $L_1$ intersect transversely (hence the number of the intersection points is finite).

Consider the space of smooth chords relative to $L_0$ as follows:
$$\cP(L_0,L_0)=\big\{x\in C^\infty([0,1], M)|\;x(0), x(1)\in L_0\big\}.$$
Clearly, the space $\cP(L_0,L_0)$ can be decomposed as follows:
\[
\cP(L_0,L_0)=\bigcup_{\eta\in \pi_0(\cP(L_0,L_0))}\cP_\eta(L_0,L_0)
\] 
where the component of a constant path in $L_0$ is denoted by $pt$. 

Given $\eta\in\pi_0(\cP(L_0,L_0))$, we choose a reference path $\gamma_\eta:[0,1]\to M$ representing the class $\eta$, and fix a symplectic trivialization $\mathfrak{t}_\eta:\gamma_\eta^*TM\to [0,1]\times\R^{2n}$ which maps the tangent space $T_{\gamma_\eta(i)}L_0$ to $ \{i\}\times\R^n\times\{0\}, i=0,1$. For each $\gamma\in \cP(L_0,L_0)$, we consider the pair $(\gamma,v)$, where $v$ is a \textit{capping} of $\gamma$, i.e., a piecewise smooth map $v:[0,1]\times [0,1]\to M$ satisfies the following conditions: $v(s,i)\in L_0$ for $i=0,1$, $v(0,t)=\gamma_\eta(t)$, and $v(1,t)=\gamma(t)$. Let us choose a trivialization $\mathfrak{t}_v$ of the pullback bundle $v^*T M$ such that the restriction of $\mathfrak{t}_v$ over $\gamma_\eta$ is the fixed trivialization $\mathfrak{t}_\eta$ as above, and that $\mathfrak{t}_v$ maps $T_{v(s,i)}L_0$ to $\R^n\times\{0\}$. Using such trivialization, to the capped path $(\gamma,v)$ one can assign a well-defined Robbin-Salamon-Maslov index $\mu_M(\gamma,v)\in \Z$;  see for instance~\cite{Us2,Go}. 

Two cappings $(\gamma_1,v_1)$ and $(\gamma_2,v_2)$ are said to be \textit{equivalent} if $\gamma_1=\gamma_2$ and the cylinder $C$ obtained by gluing $v_1$ to $v_2$ with the reversed orientation has zero $\omega$-area.

Denote by $\widetilde{\cP}_\eta(L_0,L_0)$ the cover of $\cP_\eta(L_0,L_0)$ consisting of all equivalence classes $[\gamma,v]$ of the pairs $(\gamma,v)$. For each path component $\eta$ of $\cP(L_0,L_0)$, we define the subgroup $\Gamma_\eta$ of $\R$ comprising the symplectic areas $\omega(v)$ of elements $v\in \pi_1(\cP_\eta(L_0,L_0))$ viewed as maps $v:S^1\times [0,1]\to M$ such that $v(S^1\times\{0,1\})\subset L_0$ with $v(s,\cdot)$ in $\eta$ for all $s\in 
S^1$. Let $\Gamma\subset\R$ be the \textit{complete period group} generated by the union of all $\Gamma_\eta$. 

The  \textit{Hamiltonian action} of $[\gamma,v]\in \widetilde{\cP}_\eta(L_0,L_0)$ is given by
$$\ahl([\gamma,v])=\int^1_0H(t,\gamma(t))dt-\int_{[0,1]\times [0,1]} v^*\omega.$$
We denote by $\cO_\eta(L_0,H)$ the set of all Hamiltonian chords $\gamma\in\eta$ of $H$ from $L_0$ to $L_0$, meaning that $\dot{\gamma}(t)=X_{H_t}(\gamma(t))$ for all $t\in [0,1]$, and $\gamma(0),\gamma(1)\in L_0$. It is easy to see that its cover $\widetilde{\cO}_\eta(L_0,H)$ is precisely the set consisting of the critical points of $\ahl$ on $\widetilde{\cP}_\eta(L_0,L_0)$.  Note that there is a bijection between the intersections $L_0\cap L_1$ and $\cO(L_0,H)=\bigsqcup_\eta\cO_\eta(L_0,H)$ by sending $x$ to the Hamiltonian chord $t\mapsto\varphi_H^t(x)$. 

Fix a ground field $\bK$. Throughout this paper we assume that $\bK=\Z_2$. 
In our setup for Floer theory, we will use the  Novikov field $\Lambda$ over $\bK$  given by 
\begin{equation}\label{e:Lambda}
    \Lambda=\bigg\{\sum_{i=1}^{\infty}a_iT^{\lambda_i}\;\big|\;a_i\in\bK, \;\lambda_i\in\Gamma,  \;\forall C>0, \;\#\{i\in \N|a_i\neq 0, \lambda_i<C\}<\infty \bigg\}.
\end{equation}
The \textit{valuation map} $\nu:\Lambda\to \R$ is defined as
    \[
    \nu\bigg(\sum_{i=1}^{\infty}a_iT^{\lambda_i}\bigg)=\min\{\lambda_i\;|\;a_i\neq 0\}
    \]
where we set $\nu(0)=+\infty$.

 For each $\eta\in\pi_0(\cP(L_0,L_0))$, we consider the free finite-rank complex 
 $CF_{\eta}(L_0,H)$ over $\Lambda$  freely generated by Hamiltonian chords $\gamma_i\in\cO_\eta(L_0,H)$ with fixed cappings. 
Unlike usual Floer theory, which only considers contractible Hamiltonian chords, here we need to take into account Hamiltonian chords from all free homotopy classes (we refer the reader to~\cite{CGG} for the reason of this consideration). 
In the following, we denote by $x_i$ those generators with fixed cappings in $ CF_{\eta}(L_0,H)$. Hence, 
\[
CF_{\eta}(L_0,H)=\bigoplus_i\Lambda x_i. 
\]
 
 Fix a family $\{J_t\}_{t\in[0,1]}$ of almost complex structures  with $J_t\in\cJ$ for each $t\in[0,1]$ (whenever $M$ is noncompact we also ask $J_t$ to be independent of $t$ outside of some compact codimension $0$ submanifold of $M$). 

Given $x=[\gamma_-,v_-], y=[\gamma_+,v_+]\in \widetilde{\cO}_\eta(L_0,H)$, consider the moduli space $\widehat{\cM}(x,y;H)$ of solutions $u:\R\times[0,1]\to M$ of the Floer  equation
\[\partial_s u+J_t(\partial_t u-X_{H_t}(u))=0\]
with the boundary condition $u(\R\times\{0,1\})\subset L_0$ and $u(\pm\infty,\cdot)=\gamma_\pm(\cdot)$. Denote by $\cM(x,y;H)$ the $\R$-quotient space of $\widehat{\cM}(x,y;H)$. When the moduli spaces $\widehat{\cM}(x,y;H)$ with $\mu_M(x)-\mu_M(y)=1$ are transversely cut, we then define the differential $$d_{H,J}:CF_{\eta}(L_0,H)\to CF_{\eta}(L_0,H)$$ by extending by linearity
from
$$d_{H,J}x_i=\sum_{j} \sum_v f_v T^{\omega(v)}x_j,$$
where the second sum is taken over all recappings $v$ of $x_j$ satisfying that $\mu_M(x_i)-\mu_M(x_j\#v)=1$, $f_v$ is the parity of the number of holomorphic strips  $u\in \widehat{\cM}(x_i,x_j\#v;H)$, and  the fixed capping of $x_i$ is equivalent to the capping obtained by attaching $x_j\# v$ to these strips. By a standard Gromov-Floer compactness argument, one gets $d_{H,J}\circ d_{H,J}=0$. This makes 
$(CF(L_0,H)=\oplus_{\eta} CF_{\eta}(L_0,H),d_{H,J})$ into a ungraded (or $\Z_2$-graded) complex over $\Lambda$ which restricts to a ungraded complex $CF_{\eta}(L_0,H)$ for any $\eta\in\pi_0(\cP(L_0,L_0))$.

For each $\eta$, there is a natural \textit{filtration} $\mathcal{A}$ on $CH_{\eta}(L_0,H)$ given by
\begin{equation}\label{e:filt}
\mathcal{A}\big(\sum_i\lambda_ix_i\big)=\max_i\ahl(\lambda_ix_i):=\max_i\big(\ahl(x_i)-\nu(\lambda_i)\big)
\end{equation}
where $\lambda_i\in \Lambda$ and $x_i\in \widetilde{\cO}_\eta(L_0,H)$ are generators of the complex. 

\begin{rmk}
The filtration $\mathcal{A}$ on the complex $CH_{\eta}(L_0,H)$ depends on the choice of the reference path $\gamma_\eta$. To be more specific, let $CH_\eta(L_0,H)$ be the component representing $\eta\in \pi_0(\cP(L_0,L_0))$, then  changing the reference path $\gamma_\eta$ shifts of the filtration on this component by a constant. However, this ambiguity does not affect some properties of the barcode of $CH_\eta(L_0,H)$, for instance, the bar length and the number of bars of a given length do not depend on the choice of reference paths. For more detailed discussion about this point, we refer to~\cite[Sect.3.3.2]{CGG}. 
\end{rmk}

The homology of the complex $(CH(L_0,H),d_{H,J})$ is called the \textit{Lagrangian Floer homology} of $(L_0,H)$, and is denoted by $HF(L_0,H)$ in which we suppress $J$ because it is indeed independent of $J$. Clearly, the full homology $HF(L_0,H)$ can split into the sum
\[
    HF(L_0,H) =\bigoplus_{\eta\in\pi_0(\cP(L_0,L_0))} HF_{\eta}(L_0,H).
\]
It can be shown that the homology is invariant with respect to Hamiltonians $H$. Namely, for any $G\in\cH$, we have
\[
    HF(L_0,H)\cong HF(L_0,G). 
\]
For this reason, we also use $HF(L_0,L_0)$ to denote this group to indicate that it is independent of the choice of Hamiltonians. 

Note that $HF_{\eta}(L_0,L_0)=0$ for $\eta\neq pt$ because if one chooses a $C^1$-small Hamiltonian $H\in\cH$ then all generators $[\gamma,v]$ of $CF(L_0,H)$ have $\gamma\in pt$, see~\cite[Prop.~6.2]{Us2}. Clearly, 
the full homology $HF(L_0,L_0)$ is a finite-dimensional vector space over $\Lambda$. This point will be crucial for us to compare the persistent entropy $\ph$  introduced in this paper with the barcode entropy $\bh$ defined by Çineli-Ginzburg-G\"{u}rel~\cite{CGG}.  


The filtration $\mathcal{A}$ on $CH_{\eta}(L_0,H)$ given as in $(\ref{e:filt})$ induces a non-Archimedean filtration on the Floer homology $HF_{\eta}(L_0,H)$. More precisely, for any $a\in\R$, let $CH_{\eta}^{<a}(L_0,H)$ denote the subcomplex of $CH_{\eta}(L_0,H)$ consisting of the chains $\sum_i\lambda_i x_i$ with action less than $a$, then we get a family of homologies $HF_{\eta}^{<a}(L_0,H)$ of these subcomplexes, parametrized by $a\in\R$. Moreover, we have the natural maps  
\begin{equation}\label{e:struc}
\ell_{a,b}:HF_{\eta}^{<a}(L_0,H)\longrightarrow HF_{\eta}^{<b}(L_0,H)\end{equation} 
induced by the inclusion $CH_{\eta}^{<a}(L_0,H)\hookrightarrow CH_{\eta}^{<b}(L_0,H)$ for $-\infty<a<b\leq\infty$. 


\subsection{Persistence modules and barcodes}\label{sec:persist}
In this subsection we give a quick overview of some basic facts about persistence modules and barcodes. There are various definitions of persistence modules in the literature. In the following we will use the framework which is close to the one in~\cite[Sect.4.1]{BG} (see also~\cite{BV}) for our purposes.
For more details about persistence modules we refer to the books~\cite{CCGGO,PRSZ} and the references therein.

\begin{df}\label{df:pmod}
A \textit{persistence module} with \textit{spectrum} $\mathcal{S}\subset\R$ over a field $\bK$ is a functor $(V,\pi)$ from the ordered set $(\R\setminus\mathcal{S},\leq)$ to the category of vector spaces over $\bK$, which consists of 
\begin{itemize}
    \item $\bK$-vector spaces $V_s$ for all $s\in\R$,
    \item linear maps $\pi_{s,s'}:V_s\to V_{s'}$ for each two elements $s,s'\in\R$ with $s\leq s'$,
\end{itemize}
   and which satisfies the following conditions:
\begin{itemize}
    \item[(a)] $\mathcal{S}$ is a nowhere dense subset of $\R$ which is bounded from below.
    \item[(b)] The persistence module $(V,\pi)$ is $q$-\textit{tame}: $\pi_{st}$ has finite rank for all $s<t$.
     \item[(c)] For all $t\in\R$, $V_t=\varinjlim\limits_{s<t} V_s$.
    \item[(d)]  $\pi_{s',s''}\circ\pi_{s,s'}=\pi_{s,s''}$ for all $s,s',s''\in\R$, and
if the interval $[s,s']$ does not intersect $\mathcal{S}$, then the map $\pi_{s,s'}$ is an isomorphism.
\end{itemize}
\end{df}

Let $(V,\pi)$ be a persistence module. 
Given $\delta\in\R$, we denote by $(V[\delta],\pi[\delta])$ its $\delta$-shift, i.e., a new persistence module by taking $(V[\delta])_t=V_{t+\delta}$ and $(\pi[\delta])_{s,t}=\pi_{s+\delta,t+\delta}$. 


Let $(V,\pi^V)$ and $(W,\pi^W)$ be two persistence modules.  
A \textit{morphism} $L$ from $(V,\pi^V)$ to $(W,\pi^W)$ is a family of linear maps $L_t:V_t\to W_t$, $t\in \R$, such that $\pi^W_{s,t}\circ L_s=L_t\circ\pi^V_{s,t}$ for any $s,t\in \R$ with $s\leq t$. 
For $\delta>0$, a $\delta$-shift of $L$ is defined as a morphism $L[\delta]:(V[\delta],\pi^V[\delta])\to (W[\delta],\pi^W[\delta])$ satisfying $(L[\delta])_t=L_{t+\delta}$.

\begin{df}
 Given $\delta>0$, we say that two persistence modules $(V,\pi^V)$ and $(W,\pi^W)$ are $\delta$-\textit{interleaved} if there exist two morphisms $f:V\to W[\delta]$ and $g:W\to V[\delta]$ such that $g[\delta]\circ f=\pi^V_{t,t+2\delta}$ and $f[\delta]\circ g=\pi^W_{t,t+2\delta}$.
\end{df}

 In what follows we often suppress $\pi$ in the
 notation and simply denote $(V, \pi)$ as $V$. 

 The \textit{interleaving distance} between
 two persistence modules $V$ and $W$ is defined as follows:
 \[
 d_{I}(V,W)=\inf\{\delta>0|\; V\;\hbox{and}\;W\;\hbox{are $\delta$-interleaved}\}.     
 \]
 The interleaving distance $d_I$ defines a pseudometric on the set of persistence modules, which takes value in $[0,+\infty]$; see~\cite{CCGGO}. 

\begin{thm}[Structure theorem,~{\cite{ZC}}]\label{thm:str}
 Let $V$ be a persistence module with spectrum $\mathcal{S}$. Then there is a countable collection of pairwise distinct intervals $I_1,I_2,\ldots $ of the form $(a_i,b_i]$ or $(a_i,+\infty)$ for $a_i,b_i\in \mathcal{S}$ with positive integer multiplicities $k_1,k_2,\ldots$ such that 
\begin{equation}\label{e:structure}
    V\cong \bigoplus_{i=1}^\infty
    \big(\bK I_i\big)^{k_i} 
\end{equation}
where, for $I=(a,b]$ or $I=(a,+\infty)$, $\bK I$ is a persistence module which satisfies $(\bK I)_s=\bK$ for $s\in I$ and 
$(\bK I)_s=0$ otherwise, and $\pi_{s,s'}=Id$ for $s,s'\in I$ and $\pi_{s,s'}=0$ otherwise. 
\end{thm}

\begin{df}\label{def:barcode}
The multiset which contains $k_i$ copies of $I_i$ in (\ref{e:structure}) is called the \textit{barcode} of the persistence module $V$, and is denoted by $\mathcal{B}(V)$. Each interval $I_i$ appearing in $\mathcal{B}(V)$ is called a \textit{bar}. 
\end{df}

Given a barcode $\mathcal{B}=\mathcal{B}(V)$, for a positive number $\epsilon$ and $s\in (-\infty,+\infty]$, we denote by $\mathcal{B}^\epsilon$ the multiset of all bars from $\mathcal{B}$ of length larger than $\epsilon$, and $\mathcal{B}^\epsilon_s$ the multiset of bars $(a,b]$ in $\mathcal{B}$ with $a<s$ of length $b-a>\epsilon$, respectively.  Let $\mathcal{FB}^\epsilon_s$ be the multiset of finite intervals by keeping all finite bars and truncating infinite bars at $s$ in $\mathcal{B}^\epsilon_{s-\epsilon}$, namely $(a,b]\in \mathcal{FB}^\epsilon_s$ if and only if $(a,b]$ is a finite bar in $\mathcal{B}^\epsilon_{s-\epsilon}$ or $(a,\infty)$ is an infinite bar in $\mathcal{B}^\epsilon_{s-\epsilon}$ and $b=s$.
We use $|\mathcal{B}^\epsilon|$ (resp. $|\mathcal{B}^\epsilon_s|$, $|\mathcal{FB}^\epsilon_s|$) to denote the number of bars in $\mathcal{B}^\epsilon$ (resp. $\mathcal{B}^\epsilon_s$, $\mathcal{FB}^\epsilon_s$). It is well-known that $|\mathcal{B}^\epsilon_s|<\infty$;  see~\cite{FLS,CGG2}.

For $c\in\R$, we use $\mathcal{B}[c]$ to denote a $c$-shift of $\mathcal{B}=\oplus_i(\bK I_i)^{k_i}$, i.e., $$\mathcal{B}[c]=\oplus_i(\bK(I_i-c))^{k_i}$$ where $(a,b]-c=(a-c,b-c]$ and $(a,\infty)-c=(a-c,\infty)$. 


Given an interval $I=(a,b]$ or $I=(a,+\infty)$ and a real number $\delta\in (0,\frac{b-a}{2})$ , we put $I^{\delta}=(a-\delta,a+\delta)$ or $I^{\delta}=(a-\delta,+\infty)$. 

Two barcodes $\mathcal{B}_1$ and $\mathcal{B}_2$ are said to have a $\delta$-\textit{matching} if one can delete some of bars of length less than or equal to $2\delta$ from $\mathcal{B}_1$ and $\mathcal{B}_2$ and find a bijection $F:\mathcal{B}_1^{2\delta}\to \mathcal{B}_2^{2\delta}$ such that if $F(I)=J$ then 
$I\subset J^\delta$ and $J\subset I^\delta$. The \textit{bottleneck distance} $d_B(\mathcal{B}_1,\mathcal{B}_2)$ between $\mathcal{B}_1$ and $\mathcal{B}_2$ is defined as the infimum over $\delta>0$ such that there exists a $\delta$-matching between them. Note that the bottleneck distance $d_B$ takes values in $[0,+\infty]$. 

By the isometry theorem for persistence modules (see~\cite{CSGO,BL}), the barcode map $\cB$ from Theorem~\ref{thm:str} is in fact an isometry, that is to say, for two persistence modules $V$ and $W$, we have
\[
d_{I}(V,W)=d_B(\cB(V),\cB(W)).   
\]

\subsection{Shannon entropy of persistence barcodes}

Recall that the \textit{Shannon entropy} $E(\xi)$ of a probability distribution $\xi=(p_1,\ldots, p_k)\in\R^k$ with $p_i\geq 0$ and $\sum_{i=1}^kp_i=1$ is defined as
\[
E(\xi)=-\sum_{i=1}^kp_i\log p_i; 
\]
see~\cite{Sh}.  This concept has  been introduced to topological data analysis~\cite{CGGJK,RCMP} to define a new entropy of persistent homology as follows. 

Denote by $\mathbb{B}_F$ the space of barcodes $\mathcal{B}$  which have
no infinite-length bars and $|\mathcal{B}|<\infty$, i.e., $\mathcal{B}$ consists of finite multisets comprising intervals $(a,b]$ with $-\infty< a\leq b<\infty$. 

\begin{df}\label{df:Shanentropy}
The Shannon entropy $E(\mathcal{B})$ of a persistence barcode $\mathcal{B}=\{(a_i,b_i]\}_{1\leq i\leq n}$ in $\mathbb{B}_F$ is defined as
\[
E(\mathcal{B})=-\sum_{i=1}^n\frac{\ell_i}{L}\log\bigg(\frac{\ell_i}{L}\bigg)
\]
where $\ell_i=b_i-a_i$ and $L=\sum_{i=1}^n\ell_i$.
\end{df}

\begin{rmk}
The Shannon entropy of a persistent barcode is called \textit{persistent entropy} in the literature; see for instance~\cite{AGR,AGS,CGGJK,RCMP}. Here we avoid using the term ``persistent entropy" because in this article it specifically refers to the limit of the Shannon entropy of the Floer persistence barcode under the iteration of Hamiltonian diffeomorphisms; see Sections~\ref{subsec:def-perEntropy} and~\ref{subsec:def-per-Liou}.
\end{rmk}

The Shannon entropy of a persistent barcode possesses an important stability property, which is inherited from the stability of persistent modules~\cite{CEH,CCGGO,CSGO}.
Under certain assumptions on the number and the sum of lengths of bars in $\mathcal{B}\in\mathbb{B}_F$, N. Atienza, R. Gonzalez-Diaz and M. Rucco~\cite{AGR} proved the continuity of Shannon entropy with respect to the bottleneck distance $d_B$. Furthermore, we have the following stability result. 

\begin{thm}[Stability of Shannon entropy,~{\cite[Thm~3.12]{AGS}}]\label{thm:stability}

Let $\mathcal{B}_1,\mathcal{B}_2\in\mathbb{B}_F$. If $d_B(\mathcal{B}_1,\mathcal{B}_2)\leq \frac{L_{\rm max}}{8n}$, then 
\[
|E(\mathcal{B}_1)-E(\mathcal{B}_2)|\leq \frac{4n d_B(\mathcal{B}_1,\mathcal{B}_2)}{L_{\rm max}}\bigg(\log n-\log\bigg( \frac{4n d_B(\mathcal{B}_1,\mathcal{B}_2)}{L_{\rm max}}\bigg)\bigg)
\]
where $n=|\mathcal{B}_1|+|\mathcal{B}_2|$, $L_{\rm max}=\max\{L_1,L_2\}$ and $L_i, i=1,2$ denotes the sum of lengths of all bars in $\mathcal{B}_i$.  

\end{thm}

\section{Floer persistence barcode and  boundary depth} \label{sec:floer}

In this section  we will give a way of describing a barcode of the Floer complex for the Lagrangian pair $(L_0,L_1)$ provided that $L_0$ intersects $L_1=\varphi_H^{-1}(L_0)$ transversely.

Consider the filtered Lagrangian Floer homologies
$ HF_*^{<a}(L_0,H)$, together with the natural maps $\ell$ given as in~(\ref{e:struc}). 
$ HF_*^{<a}(L_0,H)$ is neither a vector space for $a\in\R$  over $\Lambda$ nor forms  a $q$-tame persistent module over $\bK$ (cf. Definition~\ref{df:pmod}) in general. So Theorem~\ref{thm:str} does not apply to this persistence module. However, one can adapt the constructions from~\cite{UZ} to our framework and obtain a barcode  associated to the Floer persistence modules in our setting, as already done in~\cite[Sect.3.3]{CGG}. 

\begin{df} Let $V$ be a vector space over $\Lambda$, and $\mathcal{I}:V\to \R\cup\{-\infty\}$ a function. The pair $(V,\mathcal{I})$ is called a non-Archimedean normed vector space if 
\begin{itemize}
    \item[(i)] $\mathcal{I}(0)=-\infty$, and $\mathcal{I}(c)\in\R$ for all nonzero $c\in V$; 
    \item[(ii)] $\mathcal{I}(\alpha\cdot c)=\mathcal{I}(c)+\nu(\alpha)$ for all $\alpha\in \Lambda$ and $c\in V$;
    \item[(iii)] $\mathcal{I}(c_1+c_2)\leq\max\{\mathcal{I}(c_1),\mathcal{I}(c_2)\}$ for all $c_1,c_2\in V$.
\end{itemize}
\end{df}

\begin{df}
A finite-dimensional non-Archimedean
normed vector space $(V,\mathcal{I})$ over a Novikov field $\Lambda$ is called an \textit{orthogonalizable $\Lambda$-space} if there exists an orthogonal basis for $V$, meaning that a finite ordered collection $(v_1,\ldots,v_n)$ of elements of $V$ with the property
\[
\mathcal{I}\big(\sum_{i=1}^n\alpha_iv_i\big)=\max_{1\leq i\leq n}\mathcal{I}(\alpha_iv_i)
\]
for all $\alpha_1,\ldots, \alpha_n\in\Lambda$.

\end{df}

\begin{df}
Let $(V,\mathcal{I})$ be a non-Archimedean normed vector space over a Novikov field $\Lambda$. We say that two subspaces $E$ and $F$ of $V$ are orthogonal if 
\[
\mathcal{I}(e+f)=\max\big\{\mathcal{I}(e),\mathcal{I}(f)\big\}    
\]
for all $e\in E$ and $f\in F$.
\end{df}

Through a non-Archimedean Gram-Schmidt process, one can show that any subspace $U$ of an orthogonalizable $\Lambda$-space $(V,\ell)$ has an orthogonal complement $W$, i.e., $U\oplus W=V$, and $U,W$ are orthogonal; see~\cite[Corollary~2.19]{UZ}. 

Next, we record the following standard result, i.e., the \textit{singular value decomposition} for linear maps between orthogonalizable spaces.    

\begin{lem}[{\cite[Theorem~3.5]{UZ}}]\label{lem:orthogonal}
Let $(V,\ell_V)$ and $(W,\ell_W)$ be two orthogonalizable $\Lambda$-spaces. Let $L:V\to W$ be a $\Lambda$-linear map. Then there are orthogonal bases $(\xi_1,\ldots,\xi_n)$ for $V$ and $(\eta_1,\ldots,\eta_m)$ for $W$ such that the following holds:
\begin{itemize}
 \item $L\xi_i=\eta_i$ for $i=1,\ldots,r$;
    \item $(\xi_{r+1},\dots,\xi_n)$ is an orthogonal ordered basis for ${\rm ker}(L)$. 
\end{itemize}
\end{lem}

Fixing $\eta$, as a subspace of an orthogonalizable $\Lambda$-space $(CF_{\eta}(L_0,H),\mathcal{A})$, the image of the boundary map $d_{H,J}$, denoted by  $\im (d_{H,J})$, is still an orthogonalizable non-Archimedean normed vector space over $\Lambda$. Take a non-Archimedean orthogonal complement $V$ of $\im (d_{H,J})$ in $CF_{\eta}(L_0,H)$. Consider now the linear map (over $\Lambda$)
\[d_{H,J}|_V: V\longrightarrow \im (d_{H,J}).
\]
Applying Lemma~\ref{lem:orthogonal} to the above map, we see that $CF_{\eta}(L_0,H)$ has a non-Archimedean orthogonal basis $$(x_0,\ldots,x_k,y_1,\ldots,y_l,z_1,\ldots,z_l)$$ satisfying the following:
\begin{itemize}
        \item $d_{H,J} x_i=0$ for $i=1,\ldots k$;
        \item $d_{H,J}y_i=z_i$ for $i=1,\ldots,l$.
\end{itemize}

We primitively  define a barcode $\hat{\mathcal{B}}_\eta(L_0,H)$ of the Floer complex $(CF_{\eta}^{<a}(L_0,H),\ell)$ as a multiset consisting of  intervals
\[
(\mathcal{A}(z_i),\mathcal{A}(y_i)]\quad \hbox{and} \quad (\mathcal{A}(x_j),\infty)
\]
(possibly) with multiplicities for all $i,j$. 
However, there is clearly some amount of arbitrariness of this definition of barcode for the Floer complex. Firstly, the filtration $\mathcal{A}$ on the Floer complex depends on the choice of the reference path $\gamma_\eta$ as we mentioned before. Secondly,  $\hat{\mathcal{B}}_\eta(L_0,H)$ also depends on the choice of singular value decomposition which is not unique. In fact, if $$(x_0,\ldots,x_k,y_1,\ldots,y_l,z_1,\ldots,z_l)$$ is an orthogonal basis for defining $\mathcal{B}_\eta(L_0,H)$, then $$(T^{\mu_1}x_0,\ldots,T^{\mu_k}x_k,T^{\lambda_1}y_1,\ldots,T^{\lambda_l}y_l,T^{\lambda_1}z_1,\ldots,T^{\lambda_l}z_l)$$ is also a singular value decomposition for any $\mu_1,\ldots,\mu_k,\lambda_1,\ldots,\lambda_l\in\Gamma$. 

In view of the above reasons we modify the definition of a barcode (cf. Definition~\ref{def:barcode}) as follows; see~\cite[Def.~6.3]{UZ} for a different treatment.

\begin{df}\label{df:bar}
A barcode $\mathcal{B}_\eta(L_0,H)$ of the Floer complex $CF_{\eta}(L_0,H)$ is defined as the multiset of elements $([a],L)$ of $(\R/\Gamma)\times [0,\infty]$ ($L$ is called the \textit{length} of the bar $([a],L)$) comprising
\begin{itemize}
        \item infinite bars $([\mathcal{A}(x_i)],\infty)$ for $i=1,\ldots,k$;
        \item finite bars $([\mathcal{A}(z_j)],\mathcal{A}(y_j)-\mathcal{A}(z_j))$ for $j=1,\ldots,l$.
\end{itemize}
\end{df}
 It can be shown that the bar lengths and the number of bars of $\mathcal{B}_\eta(L_0,H)$ are independent of the choice of a singular value decomposition and the auxiliary data, for instance, reference paths, almost complex structures $J$ defining $d_{H,J}$ and Hamiltonians $H$ with $L_1=\varphi_H^{-1}(L_0)$; see, e.g.~\cite{PS,UZ,KS}. Note that the number of infinite bars of $\mathcal{B}_\eta(L_0,H)$ equals to $\hbox{dim}_\Lambda HF (L_0,H)$. 

\begin{rmk}
When $\Gamma=0$ (hence $\Lambda=\Z_2$),  bars $(a,b]$ and $(a,\infty)$ of a barcode in Definition~\ref{def:barcode} correspond to $(a,b-a)$ and $(a,\infty)$ of $\R\times [0,\infty]$, respectively.
\end{rmk}

The \textit{boundary depth} $\beta_\eta(L_0,L_1)$ of $CF_{\eta}(L_0,H)$ is defined as the maximal value among all finite bar-lengths  $\mathcal{A}(y_i)-\mathcal{A}(z_i)$.    
 Taking the supremum over all $\eta\in\pi_0(\cP(L_0,L_0))$ gives the boundary depth of $CF(L_0,H)$
\[
\beta(L_0,L_1)=\sup_{\eta\in\pi_0(\cP(L_0,L_0))}\beta_\eta(L_0,L_1).
\]
It was shown that 
\begin{equation}\label{e:deltaH}
|\beta(L_0,L_1)-\beta(L_0,L'_1)|\leq \delta_H(L_1,L'_1)
\end{equation}
where $L_0\pitchfork L'_1=\varphi_G^{-1}(L_0)$ for some $G\in\cH$; 
see~\cite{Us2,KS,UZ}. Therefore, in case that $L_1=\varphi_H^{-1}(L_0)$ is not transverse to $L_0$, one can define $\beta(L_0,L_1)$ as the limit of $\beta(L_0,\varphi_{H_n}^{-1}(L_0))$ for any $H_n$ with $H_n\to H$ in $C^\infty$ and  $L_0\pitchfork \varphi_{H_n}^{-1}(L_0)$.

\begin{rmk}\label{rem:lingrowth}
From (\ref{e:deltaH}) one can see that  
\[
\begin{aligned}
\beta(L_0,\varphi^{-n}(L_1))
&\leq \beta(L_0,L_0)+\delta_H(\varphi^{-n}(L_1),L_0)\\
&\leq \beta(L_0,L_0)+n\delta_H(\varphi^{-1}(L_1),L_0).
\end{aligned}
\]
for any $\varphi\in\ham(M,\omega)$ and $n\in\N$, where in the second inequality we have used the $\ham$-biinvariant property of the Hofer distance $\delta_H$. 
\end{rmk}

As in~\cite{UZ}, we define a bottleneck distance between two multisets (barcodes) $\mathcal{U}$ and $\mathcal{V}$ of elements of $(\R/\Gamma)\times [0,\infty]$ as follows:

We say that $\mathcal{U}$ and $\mathcal{V}$ have a \textit{$\delta$-matching} ($\delta>0$) if it is possible to delete the bars of length less than or equal to $\delta$ from $\mathcal{U}$ and $\mathcal{V}$ to obtain $\mathcal{U}'$ and $\mathcal{V}'$ such that there is a bijection $F:\mathcal{U}'\to \mathcal{V}'$ which satisfies that if $F([a],S)=([b],T)$  for $([a],S)\in\mathcal{U}'$ then for each $\epsilon>0$ the representative $b$ of $[b]$ can be chosen such that $|a-b|\leq\delta+\epsilon$ and either $S=T=\infty$ or $|(b+T)-(a+S)|\leq \delta+\epsilon$.  The \textit{bottleneck distance} between $\mathcal{U}$ and $\mathcal{V}$ is 
\[
d_B(\mathcal{U},\mathcal{V})=\inf\{\delta>0|\; \hbox{there exists a $\delta$-matching between $\mathcal{U}$ and $\mathcal{V}$}\}.
\]

When $\Gamma=0$, this function $d_B$ coincides with the bottleneck distance defined in Section~\ref{sec:persist}. 

In the following context we denote
$\mathcal{B}(L_0,H)=\bigsqcup_{\eta}\mathcal{B}_\eta(L_0,H)$. For $\epsilon>0$, we denote by $\mathcal{B}^\epsilon(L_0,H)$ the set of all bars from $\mathcal{B}(L_0,H)$ of length greater than $\epsilon$. Recall that $\mathcal{FB}^\epsilon(L_0,H)$ denotes the multiset of all finite-length bars from $\mathcal{B}(L_0,H)$ with length larger than $\epsilon$.
The numbers of bars for these two multisets satisfies the identity
\begin{equation}\label{e:shortbar}
|\mathcal{B}^\epsilon(L_0,H)|=|\mathcal{FB}^\epsilon(L_0,H)|+\dim_{\Lambda}HF_*(L_0,L_0).
\end{equation}
Clearly, we have
\begin{equation}\label{e:dim}
|L_0\cap\varphi_H^{-1}(L_0)|=\dim_{\Lambda}CF(L_0,H)=2|\mathcal{B}(L_0,H)|-\dim_{\Lambda}HF(L_0,L_0).
\end{equation}

For two Hamiltonians $F,G$, we define
\[
d_B(\mathcal{B}(L_0,F),\mathcal{B}(L_0,G))=\sup_\eta d_B(\mathcal{B}_\eta(L_0,F),\mathcal{B}_\eta(L_0,G)).
\]

For a barcode $\mathcal{B}=\{([a_i],l_i)|\;1\leq i\leq n\}$ with each bar $([a_i],l_i)\in(\R/\Gamma)\times [0,\infty]$, we define a \textit{shift} of $\mathcal{B}$ by a constant $c\in\R$ as
\[
\mathcal{B}[c]=\{([a_i+c],l_i)|\;1\leq i\leq n\}.
\]

The following  stability result in symplectic topology is now standard; see for instance~\cite{PS,KS,Us2,UZ}.
\begin{lem}\label{lem:stability}
Assume that $L_0$ intersects both $\varphi_F^{-1}(L_0)$ and  $\varphi_G^{-1}(L_0)$ transversely for $F,G\in\cH$. Then for each $\eta\in\pi_0(\cP(L_0,L_0))$, there is a constant $c_\eta\in\R$ such that
$$d_B(\mathcal{B}_\eta(L_0,F)[c_\eta],\mathcal{B}_\eta(L_0,G))\leq \|F\#\overline{G}\|.$$
\end{lem}

The following lemma is also well known;   see e.g.,~\cite{Us2,PS,Go}. 
\begin{lem}\label{lem:shift}
If $G\in\cH$ is another Hamiltonian such that $\varphi_G^{-1}(L_0)=\varphi_H^{-1}(L_0)$ and $HF(L_0,H)\neq 0$, then for every  $\eta\in\pi_0(\cP(L_0,L_0))$ there is a shift-isomorphism $\Psi$ from $CF_{\eta}(L_0,G)$ to $CF_{\eta}(L_0,H)$ restricting to a bijective map between $CF^{<a}_{\eta}(L_0,G)$ and $CF^{<a+c_\eta}_{\eta}(L_0,H)$ for any $a\in\R$ and some constant $ c_\eta=\mathcal{A}_{K}([\Psi\gamma_{\eta},v_{\Psi\gamma_{\eta}}])$, where $(\Psi\gamma_{\eta})(t)=\psi_t(\gamma_{\eta}(t))$,  $\psi_t=\varphi_H^t\circ(\varphi_G^t)^{-1}$, $v_{\Psi\gamma_\eta}:[0,1]^2\to M$ is a homotopy from $\gamma_{\Psi_*\eta}$ to $\Psi\gamma_{\eta}$ in $\Psi_*\eta$,  and $K\in\cH$ is given by
$K(t,x)=H(t,x)-G(t,\psi_t^{-1}(x))$.
\end{lem}
 
 \begin{rmk}\label{rmk:shift}
A direct consequence of Lemma~\ref{lem:shift} is that for each $\eta\in\pi_0(\cP(L_0,L_0))$, 
$$\mathcal{B}_\eta(L_0,H)[-c_\eta]=\{([a-c_\eta],L)|\;([a],L)\in \mathcal{B}_\eta(L_0,H)\}$$ 
defines a barcode for $CF_{\eta}(L_0,G)$. Therefore, bar length and the number of bars of $\mathcal{B}_\eta(L_0,H)$ are independent of the choice of  Hamiltonians $H$ satisfying $L_1=\varphi_H^{-1}(L_0)$ provided that $HF(L_0,L_0)\neq 0$.
\end{rmk}

\section{Persistent entropy and stability results}
In this section we introduce the persistent entropy of Floer persistence barcodes of  two Hamiltonian isotopic monotone Lagrangian submanifolds. We also present two stability results on Shannon entropy of persistence barcodes via Hofer distance.

\subsection{Persistent entropy}
\label{sec:persistent}

Given a barcode $\mathcal{B}=\{([a_i],l_i)|\;1\leq i\leq n\}$ with each bar $([a_i],l_i)\in(\R/\Gamma)\times (0,\infty)$, we define the \textit{Shannon entropy} of $\mathcal{B}$ as
\[
E(\mathcal{B})=-\sum_{i=1}^np_i\log p_i
\]
where $p_i=l_i/L$ and $L=\sum_{i=1}^nl_i$.  It is well-known that the entropy is maximized when all probabilities are equal, i.e., $p_i=1/n$, see~\cite[Appendix~7.1]{MNB}. Hence, 
\begin{equation}\label{e:max}
E(\mathcal{B})\leq \log n.
\end{equation}

The reverse estimate is false without controlling the bar lengths. The following elementary example explains the mechanism behind Question~\ref{qtn:liouville-equality}.

\begin{ex}\label{ex:algebraic-gap}
Fix $0<a<1$ and, for $m\in\N$, set
\[
\mathcal B_m=\{(0,2^m]\}\sqcup\{(0,1]\}^{\lfloor 2^{am}\rfloor}.
\]
Let $N_m=\lfloor 2^{am}\rfloor$, $A_m=2^m$, and $S_m=A_m+N_m$. For every $0<\epsilon<1$, all bars of $\mathcal B_m$ have length larger than $\epsilon$, and hence
\[
\lim_{m\to\infty}\frac{1}{m}\log |\mathcal B_m^\epsilon|
=
\lim_{m\to\infty}\frac{1}{m}\log (N_m+1)=a.
\]
We now prove that the Shannon-entropy growth is zero. The normalized length distribution consists of one weight $A_m/S_m$ and $N_m$ weights $1/S_m$. Put
\[
q_m=\frac{N_m}{S_m}.
\]
Then the total weight carried by the short bars is $q_m$, and these $N_m$ weights are equal. Therefore
\[
E(\mathcal B_m)=h(q_m)+q_m\log N_m,
\]
where $h(q)=-q\log q-(1-q)\log(1-q)$ is the binary entropy. Since
\[
q_m=\frac{N_m}{2^m+N_m}\leq 2^{-(1-a)m},
\qquad
\log N_m\leq am,
\]
we have
\[
q_m\log N_m\leq am\,2^{-(1-a)m}\to0.
\]
Moreover, for $m$ large enough $q_m\leq1/2$, and the elementary estimate
$h(q)\leq q\log(1/q)+2q$ gives
\[
h(q_m)\leq C m\,2^{-(1-a)m}
\]
for some constant $C>0$. Consequently
\[
\frac{E(\mathcal B_m)}{m}\to0.
\]
Thus exponential bar-count growth and Shannon-entropy growth can be strictly different for abstract barcode families.
\end{ex}

Assume first that $L_0$ and $L_1=\varphi_H^{-1}(L_0)$ intersect transversely. Given $\epsilon>0$, we define
\[
E^\epsilon(L_0,L_1):=E(\mathcal{F}\mathcal{B}^\epsilon(L_0,H)).
\]
Here we set $E^\epsilon(L_0,L_1)=0$ provided that $\mathcal{F}\mathcal{B}^\epsilon(L_0,H)=\emptyset$. In case that $\epsilon=0$, we set
\[
E(L_0,L_1):=E(\mathcal{F}\mathcal{B}(L_0,H)).
\]
By Remark~\ref{rmk:shift}, when $HF_*(L_0,H)\neq 0$, the entropy $E^\epsilon(L_0,L_1)$ does not depend on the choice of Hamiltonians $H$ such that $L_1=\varphi^{-1}_H(L_0)$ and $L_0\pitchfork L_1$. 

We now extend the Shannon entropy $E^\epsilon(L_0,L_1)$ to the situation where $L_0$ and $L_1$ need not  be transverse by setting
\begin{equation}\label{e:L'toL_1}
E^\epsilon(L_0,L_1)=\liminf_{L'\to L_1}E^\epsilon(L_0,L')
\end{equation}
where the limit is taken over Lagrangian submanifolds $L'$ which intersect $L_0$ transversely and are Hamiltonian isotopic to $L_1$, and converge to $L_1$ in the $C^\infty$-topology (and hence $\delta_H(L_1,L')\to 0$).

Let $K\in\cH$ be a Hamiltonian function generating $\varphi=\varphi_K^1$. Then the function
\[
K^{\sharp n}=\sum_{i=0}^{n-1}K_t\circ(\varphi_K^t)^{-i}
\]
generates the composition map $\varphi^n$, and $\overline{K}^{\sharp n}$ generates $\varphi^{-n}:=(\varphi^{-1})^n$.

\begin{df}[Relative Persistent Entropy]\label{df:relper} 
The \textit{$\epsilon$-persistent entropy of $\varphi$ relative to the pair $(L_0,L_1)$} is
\[
\ph^\epsilon(\varphi;L_0,L_1)=\limsup_{n\to\infty}\frac{E^\epsilon(L_0,\varphi^{-n}(L_1 ))}{n},
\]  
and the \textit{persistent entropy of $\varphi$ relative to $(L_0,L_1)$} is 
\[
\ph(\varphi;L_0,L_1)=\limsup_{\epsilon \searrow 0}\ph^\epsilon(\varphi;L_0,L_1)\in [0,\infty].
\]
\end{df}
Note that if $L_0\pitchfork \varphi^{-n}(L_1)$ then  $E^\epsilon(L_0,\varphi^{-n}(L_1))=E(\mathcal{F}\mathcal{B}^\epsilon(L_0,H\#K^{\sharp n}))$. 
A priori, $\ph^\epsilon(\varphi;L_0,L_1)$ could be infinite, and so does $\ph(\varphi;L_0,L_1)$. 

\begin{df}[Relative Slow Persistent Entropy]\label{df:relsper} 
The \textit{slow $\epsilon$-persistent entropy of $\varphi$ relative to the pair $(L_0,L_1)$} is
\[
\sph^\epsilon(\varphi;L_0,L_1)=\limsup_{n\to\infty}\frac{E^\epsilon(L_0,\varphi^{-n}(L_1 ))}{\log n},
\]  
and the \textit{slow persistent entropy of $\varphi$ relative to $(L_0,L_1)$} is 
\[
\sph(\varphi;L_0,L_1)=\limsup_{\epsilon\searrow0}\sph^\epsilon(\varphi;L_0,L_1)\in [0,\infty].
\]
\end{df}

\subsection{Stability of Shannon entropy}
In this subsection we record the following stability result for Shannon entropy of Floer persistence barcodes, although it is not used for the main results---mainly we are including it in order to relate Shannon entropy to Hofer geometry. 

\begin{thm}\label{thm:E-stab}
Let $L_1=\varphi^{-1}_{H_1}(L_0)$ and $L_2=\varphi^{-1}_{H_2}(L_0)$ be two  Lagrangian submanifolds of a tame symplectic manifold $(M,\omega)$ which are Hamiltonian isotopic to a closed monotone Lagrangian submanifold $L_0$  with $N_{L_0}\geq 2$ and intersect $L_0$ transversely. If $\delta_H(L_1,L_2)\leq \frac{L_{\rm max}}{8n}$, then 
\[
|E(L_0,L_1)-E(L_0,L_2)|\leq \frac{4n \delta_H(L_1,L_2)}{L_{\rm max}}\bigg(\log n-\log\bigg( \frac{4n \delta_H(L_1,L_2)}{L_{\rm max}}\bigg)\bigg)
\]
where $n=|\mathcal{FB}(L_0,H_1)|+|\mathcal{FB}(L_0,H_2)|$,  and for $i=1,2$, $L_i$ denotes the sum of lengths of all bars in $\mathcal{FB}(L_0,H_i)$ and $L_{\rm max}=\max\{L_1,L_2\}$.  
\end{thm}

From the fact that $x(\log n-\log x)$ is increasing as long as $0<x\leq \frac{n}{2}$, one can see that 
Theorem~\ref{thm:E-stab} is a direct consequence of Theorem~\ref{thm:stability} and Lemma~\ref{lem:stability}. Similarly, we have the following

\begin{thm}\label{thm:E-stab'}
Let $\varphi_1=\varphi^{1}_{H_1}$ and $\varphi_2=\varphi^{1}_{H_2}$ be two nondegenerate Hamiltonian diffeomorphisms of a closed monotone symplectic manifold $(M,\omega)$. If $d_H(\varphi_1,\varphi_2)\leq \frac{L_{\rm max}}{8n}$, then 
\[
|E(\varphi_1)-E(\varphi_2)|\leq \frac{4n d_H(\varphi_1,\varphi_2)}{L_{\rm max}}\bigg(\log n-\log\bigg( \frac{4n d_H(\varphi_1,\varphi_2)}{L_{\rm max}}\bigg)\bigg)
\]
where $n=|\mathcal{FB}(H_1)|+|\mathcal{FB}(H_2)|$, and for $i=1,2$, $L_i$ denotes the sum of lengths of all bars in $\mathcal{FB}(H_i)$ and $L_{\rm max}=\max\{L_1,L_2\}$.  
\end{thm}

\begin{rmk}\label{rmk:stability}
 Assume that Lagrangian submanifolds $L'=\varphi^{-1}_H(L_0)$ which intersect $L_0$ transversely and are Hamiltonian isotopic to $L_1$, and converge to $L_1$ in $C^\infty$-topology. Let $\mathcal{D}_\eta(L_0,L_1)\subset\R$ be the set consisting of all action differences between the capped
Hamiltonian paths in $\eta\in\pi_0(\cP(L_0,L_0))$ (cf.~\cite[Sect.3.2.3]{CGG}). Denote by $\bar{\mathcal{D}}_\eta(L_0,L_1)$  the closure of $\mathcal{D}_\eta(L_0,L_1)$ in $
 \R$, and let $\bar{\mathcal{D}}(L_0,L_1)=\cup_\eta\bar{\mathcal{D}}_\eta(L_0,L_1)$.
 
Fix $\epsilon>0$ with $\epsilon\notin \bar{\mathcal{D}}(L_0,L_1)$.
Set $b^\epsilon(L_0,L')=|\mathcal{B}^\epsilon (L_0,H)|$. In case that $L_0$ and $L_1$ are not transverse, we set
\begin{equation}\label{e:limbar}
b^\epsilon(L_0,L_1)=\liminf_{L'\to L_1} b^{\epsilon}(L_0,L').
\end{equation}
Note that regardless of whether $L_0$ and $L'$ are transverse or not,  we have
\[b^{\epsilon+\delta}(L_0,L')\leq b^\epsilon(L_0,L_1)\leq b^{\epsilon-\delta}(L_0,L')\]
whenever $\delta_H(L_1,L')<\delta/2$ (cf. the equality (4.3) in~\cite{CGG}), from which we see that $N:=|\mathcal{FB}^\epsilon(L_0,H)|\in\N$  stabilizes before the  limit~(\ref{e:limbar}). Denote by $\ell_\epsilon(L_0,L')$ the sum of lengths of all bars in $\mathcal{FB}^\epsilon(L_0,H)$. By the definition of the bottleneck distance, if $\delta_H(L_1,L')<\delta/2$ then $\ell_{\epsilon-\delta}(L_0,L')\geq \ell_{\epsilon}(L_0,L_1)-2\delta N$ (One can define the barcode of the Floer complex for  $(L_0,L_1)$ essentially by continuity in the case that the $L_0,L_1$ are not necessarily transverse,  see~page 4 of \cite{CGG} for the discussion of this point,  and hence define the quantity $\ell_{\epsilon}(L_0,L_1)$). Then, we see that  if $0< \delta<\frac{\ell_{\epsilon}(L_0,L_1)}{4N}$ 
then $\ell_{\epsilon-\delta}(L_0,L')\geq C:=\ell_{\epsilon}(L_0,L_1)/2$. Hence, $$\max\{\ell_{\epsilon}(L_0,L'),\ell_{\epsilon}(L_0,L'')\}\geq C$$
whenever $\delta_H(L_1,L'),\delta_H(L_1,L'')<\delta/2$ and $L',L''\pitchfork L_0$. Consequently, according to Theorem~\ref{thm:E-stab}, we see that $E^\epsilon(L_0,L')$ is uniformly continuous with respect to $L'$ in the Hofer distance provided that $\delta_H(L_1,L')$ is small enough. Therefore, different choices of the sequences of Lagrangian submanifolds $L'(n)$ which converges to $L_1$ in $C^\infty$-topology will result in the same limit as (\ref{e:L'toL_1}).
\end{rmk}

\begin{rmk}
One may replace the Hofer distance in Theorems~\ref{thm:E-stab} and \ref{thm:E-stab'} with the spectral metric $\gamma$ (cf.~\cite{KS})  provided that $HF(L_0,L_0)\neq 0$ since $\gamma$ bound the bottleneck distance $d_B$ from above.
\end{rmk}

\section{Barcode entropy vs. persistent entropy for Hamiltonian diffeomorphisms}\label{sec:pervsbar}

Before giving our main results on  the (relative) persistent entropy, let us recall the definition of barcode entropy given by Çineli-Ginzburg-G\"{u}rel~\cite[Def.~4.1]{CGG}. 

In the notation and conventions from
Sections~\ref{sec:floer} and \ref{sec:persistent}, given $\epsilon>0$,  when $L_0\pitchfork L_1$ with $L_1=\varphi_H^{-n}(L_0)$,  
 we set 
 \[
 b^\epsilon(L_0,L_1)=|\mathcal{B}^\epsilon (L_0,H)|
 \]
(this number depends only on the Lagrangian submanifold $L_1$ not on $H$), otherwise, we let
\[
b^\epsilon(L_0,L_1)=\liminf_{L'\to L_1} b^\epsilon(L_0,L')\]
where the limit is taken over Lagrangian submanifolds $L'$ which intersect $L_0$ transversely and are Hamiltonian isotopic to $L_1$, and converge to $L_1$ in the $C^\infty$-topology. 

\begin{df}\label{df:barEntropy}
The \textit{barcode entropy $\bh(\varphi;L_0,L_1)$ of $\varphi\in\ham(M,\omega)$ relative to $(L_0,L_1)$} is defined as 
\[
\bh(\varphi;L_0,L_1)=\limsup_{\epsilon \searrow 0}
\bh^\epsilon(\varphi;L_0,L_1)
\]
where $\bh^\epsilon(\varphi;L_0,L_1)$ is the \textit{$\epsilon$-barcode entropy}  given by 
\[
\bh^\epsilon(\varphi;L_0,L_1)=\limsup_{n\to\infty}\frac{\log^+ b^\epsilon(L_0,\varphi^{-n}(L_1))}{n}
\]
with $\log^+=\max\{\log,0\}$. 

The \textit{slow barcode entropy $\sbh(\varphi;L_0,L_1)$ of $\varphi\in\ham(M,\omega)$ relative to $(L_0,L_1)$} is defined by 
\[
\sbh(\varphi;L_0,L_1)=\limsup_{\epsilon \searrow 0}
\sbh^\epsilon(\varphi;L_0,L_1)\]
where
\[\sbh^\epsilon(\varphi;L_0,L_1)=\limsup_{n\to\infty}\frac{\log^+ b^\epsilon(L_0,\varphi^{-n}(L_1))}{\log n}.
\]

\end{df}

We remark here that $\bh^\epsilon(\varphi;L_0,L_1)$ (resp. $\sbh^\epsilon(\varphi;L_0,L_1)$) is increasing as $\epsilon \searrow 0$. Hence, one can replace $\lim_{\epsilon \searrow 0}$ with $\limsup_{\epsilon \searrow 0}$ in the definition of $\bh(\varphi;L_0,L_1)$ (resp. $\sbh(\varphi;L_0,L_1)$). 
We now prove the first result of this paper, i.e., Theorem~\ref{thm:Lpervsbar}.

\noindent\textbf{Proof of Theorem~\ref{thm:Lpervsbar}.} 
Assume first that $L_0$ and $L'=\varphi_H^{-1}(L_0)$  intersect transversely. Without loss generality, for $\epsilon>0$ we may assume that $\mathcal{F}\mathcal{B}^\epsilon(L_0,H)\neq \emptyset$.  Then by definition we have $E^\epsilon(L_0,L')=E(\mathcal{F}\mathcal{B}^\epsilon(L_0,H))$. 
Set $$\mathcal{F}\mathcal{B}^\epsilon(L_0,H)=\{([a_i],l_i)\in\R/\Gamma\times [0,\infty]\;|\;\epsilon <l_i<\infty,\;1\leq i\leq k\}.$$ 
On the one hand, by (\ref{e:shortbar}) and (\ref{e:max}), we have
\begin{eqnarray}\label{e:upbd}
E(\mathcal{F}\mathcal{B}^\epsilon(L_0,H))&\leq& \log |\mathcal{F}\mathcal{B}^\epsilon(L_0,H)|\notag\\
&=&\log(|\mathcal{B}^\epsilon(L_0,H)|-\dim_\Lambda HF_*(L_0,L_0))\notag\\
&\leq&\log |\mathcal{B}^\epsilon(L_0,H)|.
\end{eqnarray}
On the other hand, from the definition of Shannon entropy we see that 
\[
E(\mathcal{F}\mathcal{B}^\epsilon(L_0,H))=-\sum_{i=1}^k\frac{l_i}{S_k}\log \frac{l_i}{S_k}
=\log S_k-\sum_{i=1}^k\frac{l_i}{S_k}\log l_i
\]
where $S_k=\sum_{i=1}^k l_i$. Let $A_k=S_k/k$. Then $A_k,l_i\in (\epsilon,\beta(L_0,L')]$. Hence, 
\begin{eqnarray}\label{e:lowbd}
E(\mathcal{F}\mathcal{B}^\epsilon(L_0,H))&=&\log k+\log A_k-\frac{1}{S_k}\sum_{i=1}^kl_i\log l_i
\notag\\
&\geq &\log k + \log \epsilon-\frac{1}{S_k}\sum_{i=1}^k l_i\cdot\max\{\log x\;|\;x\in [\epsilon,\beta(L_0,L')]\}
\notag\\
&\geq &\log k+\log\epsilon-\max\{\log x\;|\;x\in [\epsilon,\beta(L_0,L')]\}
\end{eqnarray}
where the fact that the function $f(x)=\log x$ is increasing on $(0,\infty)$ is used in the second inequality.
Taking limit over Lagrangian submanifolds $L'\pitchfork L_0$ which are Hamiltonian isotopic to $\varphi^{-n}(L_1)$, and converge to $\varphi^{-n}(L_1)$ in the $C^\infty$-topology, from  (\ref{e:upbd}) and (\ref{e:lowbd}) we see that 
\[
E^\epsilon(L_0,\varphi^{-n}(L_1))\leq \log b^\epsilon(L_0,\varphi^{-n}(L_1)), 
\]
\[
E^\epsilon(L_0,\varphi^{-n}(L_1))\geq\log (b^\epsilon(L_0,\varphi^{-n}(L_1))-\dim_\Lambda HF(L_0,L_0)) +\log \epsilon-\max\limits_{x\in [\epsilon,\beta_n]}\log x,
\]
where $\beta_n:=\beta(L_0,\varphi^{-n}(L_1))$.
Hence, 
\begin{equation}\label{e:eupbd}
\ph^\epsilon(\varphi;L_0,L_1)\leq \bh^\epsilon(\varphi;L_0,L_1),
\end{equation}
\begin{equation}\label{e:eSupbd}
\sph^\epsilon(\varphi;L_0,L_1)\leq \sbh^\epsilon(\varphi;L_0,L_1),
\end{equation}

\begin{equation}\label{e:eplowbd}
\ph^\epsilon(\varphi;L_0,L_1)\geq \bh^\epsilon(\varphi;L_0,L_1)-\limsup_{n\to \infty} \max\limits_{x\in [\epsilon,\beta_n]}\frac{\log x}{n},
\end{equation}
\begin{equation}\label{e:eSplowbd}
\sph^\epsilon(\varphi;L_0,L_1)\geq \sbh^\epsilon(\varphi;L_0,L_1)-\limsup_{n\to \infty} \max\limits_{x\in [\epsilon,\beta_n]}\frac{\log x}{\log n}
\end{equation}
where in (\ref{e:eplowbd})  and (\ref{e:eSplowbd}) we have used  $\dim_\Lambda HF_*(L_0,L_0)<\infty$.
Note that the boundary depths $\beta_n$ grow at most linearly with the iteration $n$;  see Remark~\ref{rem:lingrowth}.   The second term in (\ref{e:eplowbd}) would vanish. So we have
\[
\ph^\epsilon(\varphi;L_0,L_1)\geq \bh^\epsilon(\varphi;L_0,L_1)
\]
which, together with the inequality (\ref{e:eupbd}), gives rise to (\ref{e:perbar}).  By (\ref{e:eSupbd}), we have the desired inequality  $(\ref{e:sperbar})$.  

Assume now that $\limsup_{n\to\infty }\frac{\log \beta_n}{\log n}=0$. We discuss in two cases: (1) if $\beta_n\leq C$ for some constant, since the function $f(x)=\log x$ is continuous on $[\epsilon,C]$ (and hence has an upper bound possibly depending on $\epsilon$), we have  
\[
\lim_{n\to\infty} \max\limits_{x\in [\epsilon,\beta_n]}\frac{\log x}{\log n}=0; 
\]
(2) if $\beta_n$ is unbounded, then we have
\[
\limsup_{n\to\infty} \max\limits_{x\in [\epsilon,\beta_n]}\frac{\log x}{\log n}=\limsup_{n\to\infty}\frac{\log 
\beta_n}{\log n}=0.
\]
Therefore, (\ref{e:eSplowbd}) gives rise to the equality
\[
\sph^\epsilon(\varphi;L_0,L_1)\geq \sbh^\epsilon(\varphi;L_0,L_1)
\]
concluding that the equality in (\ref{e:sperbar}) holds. \qed

According to Theorem~\ref{thm:Lpervsbar}, the relative barcode entropy $\bh(\varphi;L_0,L_1)$ and the relative persistent entropy $\ph(\varphi;L_0,L_1)$ are equal for any $\varphi\in\ham(M,\omega)$. Then by Proposition 4.4 in~\cite{CGG}, 
we have the following basic properties of the persistent entropy. 
\begin{prop}\label{prop:property}
The (relative) persistent entropy satisfies the following properties:

\begin{itemize}
\item[{\rm (i)}]  For each $k\in\N$, we have $\ph(\varphi^k;L_0,L_1)\leq k\ph(\varphi;L_0,L_1)$ . In particular,  $\ph(\varphi^k)\leq k\ph(\varphi)$.
\item[{\rm (ii)}] The invariant $\ph(\varphi;L_0,L_1)$ is equivariant with respect to the action of the symplectic diffeomorphism group $\symp(M,\omega)$ on $\ham(M,\omega)$ by conjugation. That is, given $\psi\in \symp(M,\omega)$,  we have
$$\ph(\psi\varphi\psi^{-1};\psi(L_0),\psi(L_1))=\ph(\varphi;L_0,L_1).$$
In particular, $\ph(\psi\varphi\psi^{-1})=\ph(\varphi)$.
\item[{\rm (iii)}] For each $\varphi\in\ham(M,\omega)$, $\ph(\varphi;L_0,L_1)=\ph(\varphi^{-1};L_0,L_1)$. In particular, $\ph(\varphi^{-1})=\ph(\varphi)$.
\item[{\rm (iv)}]  Let $L_0$ and $L_0'$ be monotone Lagrangian submanifolds of $M$ and $M'$ respectively. Let $L_1=\varphi_H^{-1}(L_0)$ and $L_1'=\varphi_{H'}^{-1}(L_0')$ for two Hamiltonians $H\in C_c^\infty([0,1]\times M,\R)$ and $H'\in C_c^\infty([0,1]\times M',\R)$. Then for any two Hamiltonian diffeomorphisms $\varphi_0: M\to M$ and $\varphi_1: M'\to M'$, we have
\[
\ph(\varphi_0\times\varphi_1;L_0\times L_0',L_1\times L_1')\leq \ph(\varphi_0;L_0,L_1)+ \ph(\varphi_1;L_0',L_1'), 
\]
and in particular, $\ph(\varphi_0\times\varphi_1)\leq \ph(\varphi_0)+\ph(\varphi_1)$. 
\item[{\rm (v)}] Fixing $\varphi\in\ham(M,\omega)$, the persistent entropy $\ph(\varphi,L_0,L_1)$ is lower semicontinuous in the pair $(L_0,L_1)$ with respect to the Hofer distance. In particular, $\ph(\varphi,L)$ is lower semicontinuous in $L$.

\end{itemize}

\end{prop}

\begin{rmk}
Note that the Floer complex of the pair $(L_0\times L_0',H\times H')$ with $L_1=\varphi_H^{-1}(L_0)$ and $L_1'=\varphi_{H'}^{-1}(L_0')$ over the universal Novikov field 
\[\Lambda_{\rm univ}=\bigg\{\sum_{i=1}^{\infty}a_iT^{\lambda_i}\;\big|\;a_i\in\bK, \;\lambda_i\in\R,  \;\forall C>0, \;\#\{i\in \N|a_i\neq 0, \lambda_i<C\}<\infty \bigg\}.
\]
 is the tensor product of the Floer complexes of $(L_0,H)$ and $(L'_0,H')$ over $\Lambda_{\rm univ}$. Note also that a coefficient extension $\Lambda \hookrightarrow \Lambda_{\rm univ}$ does not change the length of any finite bar in $\mathcal{B}(L_0,H)$; see~\cite[Prop. 6.8]{UZ}. These two facts
imply that every pair of finite bars $([a],T)\in \mathcal{B}(L_0,H)$ and $([b],T')\in \mathcal{B}(L_0',H')$ gives rise to two bars of length $\min\{T,T'\}$ in $\mathcal{B}(L_0\times L_0',H\times H')$, and that if one or both bars in a pair are infinite then the pair gives rise to one bar of length $\min\{T,T'\}$; see e.g.,~\cite{PSS,CGG}. From this the author expects that
\begin{equation}\label{e:product}
E^\epsilon (L_0\times L_0',L_1\times L_1')\geq E^\epsilon (L_0,L_1)+E^\epsilon (L_0',L_1')+\log 2 
\end{equation}
which was verified by the author for some randomly hypothesized finite bars of $\mathcal{B}(L_0,H)$ and $\mathcal{B}(L_0',H')$. Unfortunately, the author is unable to prove (\ref{e:product}). If this inequality holds true, then one can show that for any two Hamiltonian diffeomorphisms $\varphi_0: M\to M$ and $\varphi_1: M'\to M'$, 
\[
\ph(\varphi_0\times\varphi_1;L_0\times L_0',L_1\times L_1')\geq \ph(\varphi_0;L_0,L_1)+ \ph(\varphi_1;L_0',L_1').
\]
This, combining with Proposition~\ref{prop:property}~(iv), would yield
\[
\ph(\varphi_0\times\varphi_1;L_0\times L_0',L_1\times L_1')= \ph(\varphi_0;L_0,L_1)+ \ph(\varphi_1;L_0',L_1'),
\]
and in particular, $\ph(\varphi_0\times\varphi_1)= \ph(\varphi_0)+\ph(\varphi_1)$. 
\end{rmk}

\section{Persistent entropy for Liouville fillable contact manifolds}\label{sec:liouville}
In this section we shall define the \textit{persistent entropy} for any Liouville fillable contact manifold via the persistence module given by the filtered symplectic homology, and then establish a connection between barcode entropy and persistent entropy of Reeb flows. 

\subsection{Liouville domains, admissible Hamiltonians, periodic orbits and the action functional}\label{subsec:Liouville}
Let $(X,\lambda)$ be a Liouville domain with contact boundary $Y$. There exists a vector field on $X$, called the \emph{Liouville vector field} $Z$ of $(X,\lambda)$, which points transversely outward at $Y$ and satisfies
$\mathcal{L}_Zd\lambda=d\lambda$. Let $\theta=\lambda|_Y$ be the contact form on $Y$. 
The \emph{Reeb vector field} $R$ on $Y$ is defined by $d\theta(R,\cdot)=0$ and $\theta(R)=1$. The flow $\phi_\theta^t$ of $R$ is called a Reeb flow.  
The \emph{action spectrum}
$${\rm Spec}(\theta):=\bigg\{\int_\theta\lambda\ \bigg|\
\gamma\;\hbox{is a periodic Reeb orbit of}\;R\bigg\},$$
is closed and nowhere dense in $\mathbb{R}$. Moreover, $\lambda\wedge  (d\lambda)^{n-1}>0$ implies ${\rm Spec}(\theta)\subset (0,\infty)$.
The vector field $Z$ gives rise to an embedding
$\phi:(0,1]\times\partial X\to X$ such that
$$\phi(1,z)=z\quad\hbox{and}\quad \rho\partial_\rho\phi(\rho,z)=Z(\phi(\rho,z)).$$
It is easy to verify that $\phi^*\lambda=\rho\theta$, and thus $\phi^*d\lambda=d(\rho\theta)$. So a neighborhood of $\partial X$ in $X$
can be symplectically identified with the symplectic manifold $((0,1]\times\partial X, d(\rho\theta))$. Attaching a cylindrical end to
$(0,1]\times\partial X$, we obtain a \emph{completion} of $(X,\lambda)$ defined by
$$(\widehat{X},\widehat{\lambda}):=(X,\lambda)\cup_{\partial X}
([1,\infty)\times \partial X,\rho\theta).$$
Clearly, $\omega=d\widehat{\lambda}$ is a symplectic form on $\widehat{X}$. Hereafter we write $S^1=\R/\Z$. For a smooth Hamiltonian $H:S^1\times\widehat{X}\to\R$, the \emph{Hamiltonian vector field} $X_H$ associated to $H$ is defined by $-dH_t=\omega(X_{H_t},\cdot)$.
We denote $\mathcal{P}(H)\subseteq \widehat{X}$ as the set of all  $1$-periodic orbits of $X_{H_t}$:
$$\mathcal{P}(H)=\big\{z\in C^\infty(S^1,\widehat{X})\;|\; \dot{z}(t)=X_{H_t}(z(t))\}.$$

The action of a smooth loop $\gamma:S^1\to \widehat{X}$ is defined by 
\[
\mathcal{A}_H(\gamma)=\int_\gamma \widehat{\lambda}-\int_{S^1}H(\gamma(t))dt. 
\]
The set of all values of $\mathcal{A}_H$ on  $\mathcal{P}(H)$ is called the \textit{action spectrum} of $H$, denoted by $\mathrm{Spec}(H)$. 

We say that $\gamma\in\mathcal{P}(H)$ is  \textit{nondegenerate} if the linear map
$d\phi_H^1(z(0)):T_{z(0)}\widehat{X}\rightarrow T_{z(0)}\widehat{X}$ does not have $1$ as an eigenvalue. A Hamiltonian $H$ is called \textit{nondegenerate} if all $\gamma\in\mathcal{P}(H)$ are nondegenerate. 

Unless specifically stated otherwise, in this paper we mainly consider Hamiltonians $H:\widehat{X}\to\R$ which are constant on $X$ and satisfy $H=h(\rho)$ on $X\times [1,\infty)$ for some function $h\in C^\infty([1,\infty),\R)$ satisfying 
\begin{itemize}
 \item[(i)] $h$ is strictly monotone increasing;
 \item[(ii)] $h$ is convex, i.e., $h''\geq 0$ and $h''>0$ on $(1,\rho_0)$ for some $\rho_0>1$ depending on $H$;
 \item[(iii)] $h(\rho)$ is linear, i.e., $h(\rho)=\mu_H\rho-c_H$ on $(\rho_0,\infty)$ for two constants $\mu_H,c_H$ with $0<\mu_H\notin {\rm Spec}(\theta)$. In this case, we call $H$ \textit{linear at infinity}.  
\end{itemize}

We call $\mu_H$ in the (iii) the \textit{slope} of $H$.  We say that $H\in C^\infty(\widehat{X},\R)$ is \textit{admissible} if $H$ is negative on $X$ and satisfies the above conditions (i)-(iii). We denote by $\adH$ the set of all admissible Hamiltonians.

Let $J_t$ be the $t$-dependent $1$-periodic smooth \emph{$d\widehat{\lambda}$-compatible almost complex structure} on $\widehat{X}$, that is, $\langle\cdot, \cdot\rangle:=d\widehat{\lambda}(\cdot,J_t\cdot)$ is a loop of Riemannian metrics on $\widehat{X}$ and $J_t^2=-\mathrm{Id}\in \mathrm{End}(T\widehat{X})$. We call $J_t$ \emph{of contact type} on $[\rho_0,\infty)\times\partial X$ for some $\rho_0>0$ if $J_t$ is independent of $t$ outside a compact set of $\widehat{X}$ and satisfies 
$$d\rho\circ J_t=-\widehat{\lambda}\quad\hbox{on}\;[\rho_0,\infty)\times\partial X.$$
Denote by $\mathcal{J}(\widehat{X},\widehat{\lambda})$  the set consisting of almost complex structures  of contact type at infinity.

To define Floer homology of Hamiltonians linear at infinity and continuation maps between them, we need a maximum principle (cf.~\cite{Se,CFH} or~\cite[Lem.2.2]{GX} ) to guarantee that Floer cylinders
and continuation solutions of the Floer equation are contained in a compact region of $\widehat{X}$. 


\subsection{Filtered Floer and symplectic homology}\label{subsec:symphomology}
Fix a ground field $\bK$. Assume that $H$ is a nondegenerate Hamiltonian linear at infinity. Then the filtered Floer homology
$HF^s(H)$ is defined as the homology of the Floer complex $CF^s(H)$  generated by the $1$-periodic orbits with action less than $s$ provided that $s\in\R\setminus \mathrm{Spec}(H)$; see e.g.,~\cite{CFH,Vi}. When dropping the nondegenerate condition for $H$, one can still define the Floer homology of $H$ as the homology $HF^s(\widetilde{H})$ of the Floer complex $CF^s(\widetilde{H})$ of a small nondegenerate perturbation $\widetilde{H}$ of $H$ with the same slope, generated by the $1$-periodic orbits with action less than $s$. It can be shown that if $\widetilde{H}$ is sufficiently close to $H$ then $HF^s(\widetilde{H})$ is 
independent of $\widetilde{H}$. 

For $s_1,s_2\in \R\setminus \mathrm{Spec}(H)$ with $s_1<s_2$, we have the natural map 
\[
\iota_{s_1s_2}:HF^{s_1}(H)\longrightarrow HF^{s_2}(H)
\]
induced by the inclusion map $CF^{s_1}(H)\hookrightarrow CF^{s_2}(H)$. 

As a consequence, we obtain a persistent module via these homology spaces $HF^s(H)$ and these maps $\iota_{s_1s_2}$; see~\cite[Sect.3.2]{GGM}.  Indeed, one can extend the definition of $HF^s(H)$ to all $s\in\R$ by setting
\[
HF^s(H):=\varinjlim\limits HF^t(H) 
\]
where the direct limit is taken over all $t\in \R\setminus \mathrm{Spec}(H)$ with $t<s$. The map $\iota_{s_1s_2}$ can be extended naturally to these homology spaces. 
It was shown that the natural map $\iota_{s_1s_2}$ is an isomorphism provided that $[s_1,s_2)\cap \mathrm{Spec}(H)=\emptyset$; see e.g.,~\cite{Vi}.  According to Definition~\ref{df:pmod}, the family of space $s\to HF^s(H)$, together with the above natural maps, is a persistent module.  

For $s>0$, the \textit{filtered symplectic homology} $SH^s(X,\lambda)$ of $(X,\lambda)$ is defined as 
\[
SH^s(X,\lambda):=\varinjlim\limits_H HF^s(H)
\]
where the direct limit is taken over all Hamiltonians $H$ linear at infinity and
such that $H|_{S^1\times X}<0$. The symplectic homology of $(X,\lambda)$ over $\bK$ is $SH^\infty(X,\lambda)$. 
Since  admissible Hamiltonians form
a co-final family, we can restrict $H$ to the class $\adH$. Since for $s\leq 0$, $HF^s(H)=0$ for all $H\in\adH$, we have that 
$SH^s(X,\lambda)=0$ if $s\leq 0$. The map $\iota_{s_1s_2}$ from $HF^{s_1}(H)$ to $HF^{s_2}(H)$ induces the linear map
$\pi_{s_1s_2}$ from $SH^{s_1}(X,\lambda)$ to $SH^{s_2}(X,\lambda)$. Therefore, the family of filtered symplectic homology spaces $SH^s(X,\lambda)$ forms a persistence module in the sense of Definition~\ref{df:pmod} with spectrum $\mathcal{S}={\rm Spec}(\theta)\cup \{0\}$. Clearly, all bars in the barcode associated with this persistence module have  beginnings in the range $[0,\infty)$.

\subsection{Comparing barcode entropy with persistent entropy of Reeb flows}\label{subsec:per_on_Liou}
In the notation and conventions from
Sections~\ref{subsec:Liouville} and~\ref{subsec:symphomology}, we denote by $\mathcal{B}(X,\lambda)$ the barcode associated with the persistence module $SH^s(X,\lambda)$. Given $s,\epsilon>0$, we denote by $\mathcal{B}^\epsilon(X,\lambda)$ the multiset of all bars from $\mathcal{B}(X,\lambda)$ of length greater than $\epsilon$, and $\mathcal{B}^\epsilon_s(X,\lambda)$ the multiset of bars $(a,b]$ in $\mathcal{B}(X,\lambda)$ with $a<s$ of length $b-a>\epsilon$, respectively. Then we obtain a barcode $\mathcal{FB}^\epsilon_s(X,\lambda)$ comprising finite bars by  keeping finite bars and cutting those infinite bars at $s$ from $\mathcal{B}^\epsilon_{s-\epsilon}(X,\lambda)$ (cf.  Section~\ref{sec:persist}). The Shannon entropy of $\mathcal{FB}^\epsilon_s(X,\lambda)$ is denoted by $E^\epsilon_{s}(X,\lambda)$. 
Clearly, we have $|\mathcal{FB}^\epsilon_s(X,\lambda)|=|\mathcal{B}^\epsilon_{s-\epsilon}(X,\lambda)|$. 

The \textit{persistent entropy} of $(X,\lambda)$ is by definition the exponential growth rate
\[
\ph(X,\lambda):=\limsup_{\epsilon \searrow 0}\limsup_{s\to\infty}\frac{ E^\epsilon_{s}(X,\lambda)}{s}.
\]
See also Definition~\ref{df:per_Liou}. The \textit{slow persistent entropy} $\sph(X,\lambda)$ of $(X,\lambda)$ is defined similarly by replacing the linear growth of  $E^\epsilon_{s}(X,\lambda)$ with its logarithmic growth.

Recall that the barcode entropy of a Liouville domain $(X,\lambda)$, denoted by $\bh(X,\lambda)$, measures the exponential growth rate of $b^\epsilon_s(X,\lambda)=|\mathcal{B}^\epsilon_s(X,\lambda)|$; see~\cite{CGG2,FLS}.
\begin{df}\label{df:bar_Liou}
The \textit{barcode entropy} of $(X,\lambda)$ is defined as 
\[
\bh(X,\lambda)=\lim_{\epsilon \searrow 0}
\bh^\epsilon(X,\lambda)
\]
where $\bh^\epsilon(X,\lambda)$ is the \textit{$\epsilon$-barcode entropy} given by 
\[
\bh^\epsilon(X,\lambda)=\limsup_{s\to\infty}\frac{\log^+ b^\epsilon_s(X,\lambda)}{s}.
\]
where the logarithm is taken base 2 and $\log^+=\log$ except that $\log^+0=0$.
\end{df} 
The \textit{slow barcode entropy} of $(X,\lambda)$ is defined as the polynomial growth rate
\[
\sbh(X,\lambda)=\lim_{\epsilon \searrow 0}\sbh^\epsilon(X,\lambda)=\lim_{\epsilon \searrow 0}\limsup_{s\to\infty}\frac{\log^+ b^\epsilon_s(X,\lambda)}{\log s}.
\]

The \textit{dynamics barcode entropy} of $(X,\lambda)$ is defined by
\[
\hbar_{Dyn}^\epsilon(X,\lambda)=\limsup_{s\to\infty}\frac{\log^+ b^\epsilon_{s-\epsilon}(X,\lambda)}{s}
\quad\hbox{and}\quad \hbar_{Dyn}(X,\lambda)=\lim_{\epsilon \searrow 0}
\hbar_{Dyn}^\epsilon(X,\lambda).
\]
The above definition was given in~\cite{CGG2}.  By analogy with the slow barcode entropy of $(X,\lambda)$, we define
\[
\hbar_{\mathcal{S}Dyn}^\epsilon(X,\lambda):=\limsup_{s\to\infty}\frac{\log^+ b^\epsilon_{s-\epsilon}(X,\lambda)}{\log s},
\quad
\hbar_{\mathcal{S}Dyn}(X,\lambda):=\lim_{\epsilon \searrow 0}\hbar_{\mathcal{S}Dyn}^\epsilon(X,\lambda).
\]
Note that $b^\epsilon_{s-\epsilon}(X,\lambda)$ represents the number of  bars of length greater than $\epsilon$ and beginning below $s-\epsilon$ in the barcode $\mathcal{B}(X,\lambda)$. It is easy to see that 
\begin{equation}\label{e:dyn=bar}
\bh(X,\lambda)=\hbar_{Dyn}(X,\lambda)\quad\hbox{and}\quad \sbh(X,\lambda)=\hbar_{\mathcal{S}Dyn}(X,\lambda);
\end{equation}
see also~\cite[Prop.4.3]{CGG2}.

To prove Theorem~\ref{thm:liouville}, we need the following well-known result.

\begin{lem}[{\cite[Thm.2.3]{CGG0}}]\label{lem:upbdbd}
Let $(X,\lambda)$ be a Liouville domain with $SH(X,\lambda)=0$. Then the lengths of all bars in the barcode $\mathcal{B}(X,\lambda)$ have a uniform upper bound.
\end{lem}

The same argument gives the following useful criterion. For $s,\epsilon>0$, let $U_s^\epsilon(X,\lambda)$ be the maximal length of a bar in $\mathcal{FB}_s^\epsilon(X,\lambda)$; if this barcode is empty, set $U_s^\epsilon=\epsilon$.

\begin{prop}[Subexponential length-growth criterion]
\label{prop:subexp-length}
Let $(X,\lambda)$ be a Liouville domain. If, for every $\epsilon>0$,
\[
\limsup_{s\to\infty}\frac{\log U_s^\epsilon(X,\lambda)}{s}=0,
\]
then $\ph(X,\lambda)=\bh(X,\lambda)$. If, moreover,
\[
\limsup_{s\to\infty}\frac{\log U_s^\epsilon(X,\lambda)}{\log s}=0
\]
for every $\epsilon>0$, then $\sph(X,\lambda)=\sbh(X,\lambda)$.
\end{prop}

\begin{proof}
Write $\mathcal{FB}_s^\epsilon(X,\lambda)=\{(a_i,b_i]\}_{i=1}^{N_s^\epsilon}$ and $l_i=b_i-a_i$. Then $N_s^\epsilon=b^\epsilon_{s-\epsilon}(X,\lambda)$ and $\epsilon<l_i\leq U_s^\epsilon$. The maximality of Shannon entropy gives
\[
E_s^\epsilon(X,\lambda)\leq \log N_s^\epsilon.
\]
Conversely, with $S=\sum_i l_i$ and $p_i=l_i/S$, one has $p_i\leq U_s^\epsilon/(N_s^\epsilon\epsilon)$, and hence
\[
E_s^\epsilon(X,\lambda)=-\sum_i p_i\log p_i
\geq \log N_s^\epsilon+\log \epsilon-\log U_s^\epsilon.
\]
Dividing by $s$ gives $\ph^\epsilon(X,\lambda)=\hbar_{Dyn}^\epsilon(X,\lambda)$ under the first assumption; dividing by $\log s$ gives the slow version under the second assumption. Letting $\epsilon\searrow0$ and using~(\ref{e:dyn=bar}) completes the proof.
\end{proof}

\noindent\textbf{Proof of Theorem~\ref{thm:liouville}.} 
The inequalities $\ph(X,\lambda)\leq \bh(X,\lambda)$ and $\sph(X,\lambda)\leq \sbh(X,\lambda)$ follow from the estimate $E_s^\epsilon(X,\lambda)\leq \log |\mathcal{FB}_s^\epsilon(X,\lambda)|=\log b^\epsilon_{s-\epsilon}(X,\lambda)$ and~(\ref{e:dyn=bar}). If $SH(X,\lambda)=0$, Lemma~\ref{lem:upbdbd} gives a uniform upper bound on $U_s^\epsilon(X,\lambda)$, so Proposition~\ref{prop:subexp-length} gives the reverse inequalities. 
\qed

\section{Flexibility of barcode and persistent entropies for Reeb flows}

Let $\theta$ be a contact form on a manifold $Y$ of dimension $2n-1\geq 3$. 
The contact volume of $Y$ with respect to $\theta$ is defined as
\[
\mathrm{vol}_\theta(Y)=\frac{1}{n!\omega_n}\int_Y\theta\wedge(d\theta)^{n-1}. 
\]

The key to prove Theorem~\ref{thm:collapse} is the following entropy collapse of Reeb flows proved by Abbondandolo-Alves-Sağlam-Schlenk:
\begin{thm}[{\cite[Thm.1.1]{AASS}}]\label{thm:aass}
Let $(Y,\xi)$ be a closed co-orientable contact manifold. For every $\varepsilon>0$, $(Y,\xi)$ admits a contact form $\theta$ which satisfies $\mathrm{vol}_\theta(Y)=1$ and that the topological entropy $h_{\mathrm{top}}(\theta)$ of its Reeb flow is smaller than $\varepsilon$. 

\end{thm}

\noindent\textbf{Proof of Theorem~\ref{thm:collapse}.} 
Suppose that there exists a Liouville form $\lambda_0$ on the Liouville domain $X$ with $\partial X = Y$ and $\ker(\lambda_0|_Y)=\xi$. Let $\theta$ be a contact form on $Y$ with $\ker\theta = \xi$ which is given by Theorem~\ref{thm:aass}.  We will show that there exists a Liouville 1-form $\lambda$ on $X$ such that $\lambda|_Y = \theta$. Although this fact is, of course, well-known to experts, we include its proof for the sake of completeness.

Since $\partial X = Y$ is a contact boundary, there exists a collar neighbourhood $$Y\times(-\varepsilon,0]\subset X$$ with collar coordinate $t$ ($t=0$ at the boundary) such that in this neighbourhood
\[
\lambda_0 = e^t\alpha,
\]
where $\alpha = \lambda_0|_Y$ is the original contact form on $Y$. Write $\theta = f\alpha$ with $f:Y\to\mathbb{R}_{>0}$ a smooth positive function. Since $Y$ is compact, $f$ has a positive minimum. Choose a constant
\[
0<c<m:=\min_Y f>0. 
\]
Choose a smooth function
$\beta:(-\varepsilon,0]\to[0,1]$
such that
$\beta(t)=0$ near  $t=-\varepsilon$, 
$\beta(t)=1$ near  $t=0$, 
and $\beta'(t)\ge 0$. We define
\[
h(t,x):=c+\beta(t)(f(x)-c).
\]

Now we construct a new $1$-form $\lambda$ on $X$ as follows:
\begin{itemize}
\item On the collar $Y\times(-\epsilon,0]\subset X$, define 
\[
\lambda:=e^t h(t,x)\alpha;
\]
\item On the complement of $Y\times(-\epsilon,0]$ in $X$, define
\[
\lambda:=c\lambda_0.
\]
\end{itemize}
These definitions agree near the inner end of the collar because
$\beta=0$ there, so $\lambda$ is a globally smooth one-form on $X$.

At the boundary $t=0$, since $\beta=1$ near $0$, we have
\[
\lambda|_Y
=
h(0,x)\alpha
=
f(x)\alpha
=
\theta.
\]
It remains to check that $d\lambda$ is symplectic. This is automatic away
from the collar, since there
$d\lambda=c\,d\lambda_0$
is symplectic. Hence it suffices to check the collar.

Write
\[
r(t,x):=e^th(t,x).
\]
Then
$\lambda=r\alpha$ on the collar,
and therefore
\[
d\lambda=dr\wedge\alpha+r\,d\alpha.
\]
Suppose that $\mathrm{dim} Y=2n-1$. Using 
$\alpha\wedge\alpha=0$ and the fact that
$(d\alpha)^n$ vanishes on $Y$ for dimensional reasons, we obtain
\[
(d\lambda)^n
=
n r^{n-1} dr\wedge \alpha\wedge(d\alpha)^{n-1}.
\]
Write $dr=(\partial_t r)\,dt+d_Yr$. The term involving $d_Yr$ vanishes in the top wedge product, because
$
d_Yr\wedge\alpha\wedge(d\alpha)^{n-1}
$
is a $2n$-form purely in the $Y$-directions, while $Y$ has dimension $2n-1$. Hence
\[
(d\lambda)^n
=
n r^{n-1}(\partial_t r)\,
dt\wedge\alpha\wedge(d\alpha)^{n-1}.
\]
Since
$h(t,x)\ge c+\beta(t)(\min_Y f-c)>0$ and $\partial_th=\beta'(t)(f(x)-c)\ge 0$, we immediately obtain
\[
\partial_t r=e^t(h+\partial_th)>0.
\]
Consequently, 
$(d\lambda)^n\neq0$
everywhere on the collar. Hence $d\lambda$ is symplectic on all of $X$.
Moreover, by construction,
$\lambda|_Y=\theta$.
Thus $(X,\lambda)$ is a Liouville domain filling of $(Y,\xi)$. 

To finish the proof, using Stokes' formula we get
\[
\mathrm{vol}_\theta(Y)=\mathrm{vol}_{d\lambda}(X).
\]
Since $\mathrm{vol}_\theta(Y)=1$ and $h_{\mathrm{top}}(\theta)<\varepsilon$, combining the first inequality in Theorem~\ref{thm:liouville} and the inequality (\ref{eq:bartop}) concludes the proof of the theorem. \qed

Let  $(M,g)$ be a closed Riemannian manifold. 
The space of free loops $\gamma: S^1:=\mathbb{R}/\mathbb{Z}\to M$ is denoted by $\Lambda M$. The \textit{length} and \textit{energy} of a loop $\gamma\in \Lambda M$ are given by
\[
\ell(\gamma)=\int_0^1|\dot{\gamma}(t)|_gdt,\qquad E(\gamma)=\int_0^1|\dot{\gamma}(t)|_g^2dt.
\]
The energy defines a smooth Morse-Bott function on the loop space whose
critical points are constant loops and geodesics in $\Lambda M$. We denote by $\mathcal{P}_g$ the set of closed geodesics $\gamma\in \Lambda M$ which are geometrically different, meaning that two closed geodesics that differ by a reparameterization are regarded as the same geodesic. Then we have
\[
\mathcal{P}_g=\bigcup_{t>0}\mathcal{P}_g^t,\qquad \mathcal{P}_g^t=\{\gamma\in \mathcal{P}_g\;|\;\ell(\gamma)<t\}.
\]

Let $\lambda_\mathrm{can}$ be the canonical Liouville $1$-form on $T^*M$ which  has the form $pdq$ in local coordinates $(q,p)$ of $T^*M$. 
The unit disk cotangent of $M$ is defined as 
\[
D_g^*M:=\{(q,p)\in T^*M\;|\; |p|_g\leq 1\}
\]
with boundary $S_g^*M$.

The filtered symplectic homology $SH_*^s(D_g^*M,\lambda_\mathrm{can})$ over $\Z_2$, together with linear maps $\pi_{s,t}$, forms a persistent module; see~Section~\ref{subsec:symphomology}. 
We first record the barcode computation for negatively curved cotangent disk bundles.

\begin{lem}\label{lem:sympbar}
Let  $(M^n,g)$ be a closed Riemannian manifold of negative curvature.  Then the barcode of the persistent module $(SH_*^s(D_g^*M,\lambda_\mathrm{can}),\pi)$ has the form
\[
\mathcal{B}(D_g^*M,\lambda_\mathrm{can})=\bigoplus_{k=0}^n(\Z_2(0,\infty))^{b_k(M)}\oplus\bigoplus_{c\in \mathcal{P}_g} \big(\Z_2(\ell(c),\infty))^2
\]
where $b_k(M)$ denotes the dimension of the $k$-th homology of $M$ over $\Z_2$. 

\end{lem}

\begin{proof}
According to Theorem~5.3 in~\cite{CHO}, we have   isomorphisms on filtered symplectic homology
\[
\Psi: SH_*^s(D_g^*M,\lambda_\mathrm{can})\longrightarrow MH_*^s(\sqrt{E})
\]
where  $MH_*^s(\sqrt{E})$ denotes the  homology of the Morse chain complex $CM_*^s(\sqrt{E})$ of $E:\Lambda M\to \R$,  graded by the Morse indices of $E$, but filtered by the square root $\sqrt{E}$. Moreover, 
for $t>s>0$, the following  diagram commutes:
\[
\begin{tikzcd}
SH_*^s(D_g^*M,\lambda_\mathrm{can}) \arrow[r, "\Psi"] \arrow[d, "\pi_{s,t}"'] & MH_*^s(\sqrt{E}) \arrow[d, "\iota_{s,t}"] \\
SH_*^t(D_g^*M,\lambda_\mathrm{can}) \arrow[r, "\Psi"] & MH_*^t(\sqrt{E})
\end{tikzcd}
\]
with the map $\iota_{s,t}$ induced by the natural inclusion map on the chain level. As a consequence, the two persistence modules given by the filtered symplectic homology and filtered Morse homology of loop space are isomorphic, and hence, the corresponding barcodes coincide. It suffices to compute the barcode of the latter. 

According to the free homotopy classes, we decompose the free loop space $\Lambda M$ as
\[
\Lambda M=\bigsqcup_{\alpha\in\tilde{\pi}(M)}\Lambda_\alpha M. 
\]
Since $M$ has negative sectional curvature, each noncontractible component $\Lambda_\alpha M$ contains a unique closed geodesic $c_\alpha$ up to a time shift. Set
\[
S^1\cdot c_\alpha=\{c_\alpha(\cdot+s)\;|\; s\in S^1\}. 
\] 
The length functional attains a minimum at  $c_\alpha$, i.e., 
$$\ell_\alpha:=\ell_g(c_\alpha)=\inf_{x\in \Lambda_\alpha M}\ell(x).$$
The energy functional $E$ on $\Lambda_\alpha M$ has the unique Morse-Bott critical submanifold $S^1\cdot c_\alpha$ with Morse index $0$.  

Set 
\[
\Lambda_\alpha^{<t} M=\{x\in \Lambda_\alpha\;|\; \sqrt{E}(x)<t\}.
\]
Clearly, $\Lambda_\alpha^{<t} M=\emptyset$ for $t\leq\ell_\alpha$. The standard deformation lemma (cf.~\cite{Wan} or~\cite[Lemma~16]{Go0}) implies that $\Lambda_\alpha^{<t} M$ is homotopy equivalent to $S^1\cdot c_\alpha$, i.e., 
\[
\Lambda_\alpha^{<t} M\simeq S^1\cdot c_\alpha\cong S^1
\]
for any $t>\ell_\alpha$. 

Since our manifold $M$ of negative curvature has no nonconstant contractible closed geodesics, the energy functional $E$ has the unique critical submanifold consisting of constant closed geodesics in $M$ 
on the contractible component $\Lambda_0 M$. So we get
\[
\Lambda_0^{<t} M\simeq M
\]
for any $t>0$. 

Combining the free homotopy classes above, we obtain

\[
\Lambda^{<t} M = \Lambda_0^{<t} M \sqcup \bigsqcup_{\begin{subarray}{c}\alpha \neq 0 \\ \ell_{\alpha} < t \end{subarray}} \Lambda_{\alpha}^{<t} M.
\]
Hence,
\[
H_*(\Lambda^{<t} M; \mathbb{Z}_2) \cong H_*(M; \mathbb{Z}_2) \oplus \bigoplus_{\begin{subarray}{c}\alpha \neq 0 \\ \ell_{\alpha} < t \end{subarray}} H_*(S^1; \mathbb{Z}_2).
\]
For $0<s<t$, the homomorphism induced  by the natural inclusion map
$\Lambda^{<s}M\hookrightarrow \Lambda^{<t}M$
\[
\iota_{s,t}:
H_*(\Lambda^{<s}M;\Z_2)
\longrightarrow
H_*(\Lambda^{<t}M;\Z_2)
\]
is the identity on the summand $H_*(M;\Z_2)$, and it is the identity on each summand
$H_*(S^1;\Z_2)$
corresponding to a free homotopy class \(\alpha\neq 0\) with
$\ell_\alpha<s$. 
If $s\leq \ell_\alpha<t$,
then the summand corresponding to \(\alpha\) is born at filtration value \(\ell_\alpha\).

Consequently, every bar in the barcode of the persistent module $t\mapsto H_*(\Lambda^{<t}M;\Z_2)$ is infinite. More precisely, once a homology class appears at some filtration value, it persists for all larger filtration values.
Indeed, the summand \(H_*(M;\Z_2)\) is already present for all \(t>0\), coming from the contractible component of the free loop space. For each nontrivial free homotopy class \(\alpha\), a new connected component enters the sublevel set precisely when \(t\) passes the length \(\ell_\alpha\) of the unique closed geodesic in that class. For all larger values of \(t\), this component remains homotopy equivalent to \(S^1\). Hence the two homology classes contributed by
$H_*(S^1;\Z_2)$
are born at \(t=\ell_\alpha\) and never die.
Therefore, 
 the barcode of the persistence module
$t\longmapsto H_*(\Lambda^{<t}M;\Z_2)$ 
is given by
\[
\operatorname{Bar}\bigl(H_*(\Lambda^{<\bullet}M;\Z_2)\bigr)
=
(\Z_2(0,\infty))^{b(M)}
\oplus
\bigoplus_{\alpha\neq 0}
\left(
\Z_2(\ell_\alpha,\infty)_0
\oplus
\Z_2(\ell_\alpha,\infty)_1
\right).
\]
Here \(b(M)=\sum_{k=0}^n\dim H_k(M;\Z_2)\), and the subscripts indicate homological degrees. This gives rise to the barcode $\mathcal{B}(D_g^*M,\lambda_\mathrm{can})$.
\end{proof}

\begin{prop}\label{prop:negcurv-cotangent}
Let $(M^n,g)$ be a closed Riemannian manifold of negative curvature, and let $\varphi_g^t$ be its geodesic flow. Then
\[
\ph(D_g^*M,\lambda_\mathrm{can})=\bh(D_g^*M,\lambda_\mathrm{can})=h_\mathrm{top}(\varphi_g).
\]
\end{prop}

\begin{proof}
By Lemma~\ref{lem:sympbar}, the barcode of $SH_*^s(D_g^*M,\lambda_\mathrm{can})$ has only infinite bars. Hence every bar in $\mathcal{FB}_s^\epsilon(D_g^*M,\lambda_\mathrm{can})$ has length at most $s$, and Proposition~\ref{prop:subexp-length} gives
\[
\ph(D_g^*M,\lambda_\mathrm{can})=\bh(D_g^*M,\lambda_\mathrm{can}).
\]
Moreover, Lemma~\ref{lem:sympbar} gives
\[
|\mathcal B_s^\epsilon(D_g^*M,\lambda_\mathrm{can})|=b(M)+2\#\mathcal P_g^s.
\]
Margulis' theorem~\cite{Ma} implies
\[
h_\mathrm{top}(\varphi_g)=\lim_{s\to\infty}\frac{\log\#\mathcal P_g^s}{s},
\]
and therefore $\bh(D_g^*M,\lambda_\mathrm{can})=h_\mathrm{top}(\varphi_g)$.
\end{proof}

\noindent\textbf{Proof of Theorem~\ref{thm:flexibility}.} 
By~\cite[Thm.A]{EK}, for every $c>2\sqrt{\pi(k-1)}$  there exists a negatively curved Riemannian metric $g$ on $S$ with total area $1$ and $h_\mathrm{top}(\varphi_g)=c$. For the canonical Liouville form $\lambda$ on $T^*S$, the Reeb flow of $\lambda|_{S_g^*S}$ is the co-geodesic flow of $g$, and $\mathrm{vol}_{d\lambda}(D_g^*S)=1$. Proposition~\ref{prop:negcurv-cotangent} gives
\[
\ph(D_g^*S,\lambda)=\bh(D_g^*S,\lambda)=h_\mathrm{top}(\varphi_g)=c.
\]
This completes the proof. \qed

\noindent\textbf{Proof of Corollary~\ref{cor:surface-endpoint}.}
By~\cite{GGM}, the barcode entropy of the geodesic flow on a closed surface agrees with its topological entropy. Thus
\[
\bh(D_g^*S,\lambda_{\mathrm{can}})=h_{\mathrm{top}}(\varphi_g).
\]
Katok's entropy rigidity theorem~\cite{Ka}, with the above volume normalization, gives
\[
h_{\mathrm{top}}(\varphi_g)\geq 2\sqrt{\pi(k-1)},
\]
and equality holds precisely for constant negative curvature metrics. If $g$ has constant negative curvature, then Proposition~\ref{prop:negcurv-cotangent} gives
\[
\ph(D_g^*S,\lambda_{\mathrm{can}})=\bh(D_g^*S,\lambda_{\mathrm{can}}),
\]
and the desired endpoint equality follows. \qed


\begin{thebibliography}{SK}

\bibitem{AASS}
A. Abbondandolo, M.R. Alves,  M. Sağlam and F. Schlenk, 
Entropy collapse versus entropy rigidity for Reeb and Finsler flows.
 {\it  Selecta Math. (N.S.)} {\bf 29} (2023),  Paper No. 67, 99 pp.

\bibitem{Ah} J. Ahn, Barcode entropy and relative symplectic cohomology. Preprint arXiv:2602.07220.

\bibitem{AGR} N. Atienza, R. Gonzalez-Diaz and M. Rucco, Persistent entropy for separating topological features from noise in vietoris-rips complexes, 
{\it  J. Intell. Inf. Syst.} {\bf  52}, no. 3 (2019), pp. 637--655.

\bibitem{AGS} N. Atienza, R. Gonzalez-Diaz and M. Soriano Trigueros, On the stability of persistent entropy and new summary functions for topological data analysis,   {\it  Pattern Recognition} {\bf  107} (2020), 107509.

\bibitem{ALP} M. Audin,  F. Lalonde and L. Polterovich,   \textit{Symplectic rigidity: Lagrangian submanifolds,} from:``Holomorphic curves in symplectic geometry", (M. Audin and J. Lafontaine, editors), Progr. Math., {\bf 117}, Birkh\"{a}user, Basel (1994) 271--321.

\bibitem{BG} E. Barut and V.L. Ginzburg, Barcode growth for toric-integrable Hamiltonian systems. 
Preprint arXiv:2503.08922.

\bibitem{BCG} G. Besson, G. Courtois and S. Gallot, Entropies et rigidit\'{e}s des espaces localement sym\'{e}triques de courbure strictement n\'{e}gative. {\it Geom. Funct. Anal.} {\bf 5} (1995), no. 5, 731--799.

\bibitem{BL} U. Bauer and M. Lesnick, Induced matchings and the algebraic stability of persistence barcodes, {\it  J. Comput. Geom.} {\bf 6} no. 2 (2015), 162--191. 

\bibitem{BMRPV} J. Binchi, E. Merelli, M. Rucco, G. Petri, and F. Vaccarino,  jholes: A tool for understanding
biological complex networks via clique weight rank persistent homology.  {\it Electron. Notes Theor. Comput. Sci., vol.} {\bf 306} (2014), 5--18.

\bibitem{BV} P. Bubenik, T.Vergili, Topological spaces of persistence modules and their properties.
 {\it J. Appl. Comput. Topol., } {\bf 2} (2018), 233--269.

\bibitem{BP3S2} L. Buhovsky, J. Payette, I. Polterovich, L. Polterovich, E. Shelukhin and V. Stojisavljevi\'{c}, Coarse nodal count and topological persistence. Preprint arXiv:2206.06347. 



\bibitem{CCGGO} F. Chazal, D. Cohen-Steiner, M. Glisse, L.J. Guibas, and S.Y. Oudot,
Proximity of persistence modules and their diagrams. In \textit{Proceedings of
the Twenty-fifth Annual Symposium on Computational Geometry}, SCG
'09, pages 237--246, New York, NY, USA, 2009. ACM.

 \bibitem{CSGO} F. Chazal, V. de Silva, M. Glisse, and S. Oudot, {\it The structure and stability of persistence modules.} SpringerBriefs Math. Springer, [Cham], 2016, x+120 pp.

 \bibitem{Ch} Y.-V. Chekanov, Invariant Finsler metrics on the space of Lagrangian embeddings. {\it Math. Z.} {\bf 234} (2000),  605--619.

\bibitem{CGGJK} H. Chintakunta, T. Gentimis, R. Gonzalez-Diaz, M. J. Jimenez, and H. Krim, An entropy
based persistence barcode.  {\it Pattern Recognit.} {\bf 48} (2015),  391--401.

\bibitem{CFH} K. Cieliebak, A. Floer, H. Hofer, Symplectic homology II: A general construction,  {\it Math.
Zeit.} {\bf 218} (1995), 103--122.

\bibitem{CHO} 
K. Cieliebak, N. Hingston and A. Oancea, 
Loop coproduct in Morse and Floer homology. 
{\it J. Fixed Point Theory Appl.} {\bf 25} (2023), Paper No. 59, 84 pp.

 \bibitem{Cin}  E. Çineli, A generalized pseudo-rotation with positive topological entropy.  {\it Bull. Lond. Math. Soc.}  {\bf 57} (2025), no. 4, 1140--1149.

 \bibitem{CGG0} E. Çineli, V.L. Ginzburg and B.Z. G\"{u}rel, Closed Orbits of Dynamically Convex Reeb Flows: Towards the HZ- and Multiplicity Conjectures. Preprint arXiv:2410.13093. 
 
 \bibitem{CGG} E. Çineli, V.L. Ginzburg and B.Z. G\"{u}rel, Topological entropy of Hamiltonian diffeomorphisms: a persistence homology and Floer theory perspective. {\it  Math. Z. }  {\bf 308} (2024), no. 4, Paper No. 73, 38 pp.
 
 \bibitem{CGG2} E. Çineli, V.-L. Ginzburg and B.-Z. G\"{u}rel, On the Barcode Entropy of Reeb Flows.  {\it Selecta Mathematica, New Series}  {\bf 31} (2025), no. 4, Paper No. 64.
 
 \bibitem{CGG3} E. Çineli, V.-L. Ginzburg and B.-Z. G\"{u}rel, On the growth of the Floer barcode. {\it J. Mod.
Dyn. }  {\bf  20} (2024), 275--298.
 
 \bibitem{CGG4} E. Çineli, V.-L. Ginzburg and B.-Z. G\"{u}rel, From Barcode Entropy to Metric Entropy. Preprint arXiv:2507.13215.

 
  \bibitem{CEH} D. Cohen-Steiner, H. Edelsbrunner, and J. Harer. Stability of persistence diagrams. In \textit{ Proc. 21st ACM Sympos. Comput. Geom.} (2005),  263--271.
 
\bibitem{EK} A. Erchenko and A. Katok. Flexibility of entropies for surfaces of negative curvature. {\it Israel J. Math.} {\bf  232} (2019) 631--676.
 
 \bibitem{FLS} E. Fender, S. Lee and B. Sohn, Barcode entropy for Reeb flows on contact manifolds with exact Liouville fillings, {\it Comm. Contemp. Math.} (2025). doi.org/10.1142/S0219199725500440. 
 
 \bibitem{Fer0} R. Fernandes, Barcode entropy and wrapped Floer homology. Preprint arXiv:2410.05528.
 
  \bibitem{Fer1} R. Fernandes, Wrapped Floer homology and hyperbolic sets. Preprint arXiv:2501.06654.
  

  
 \bibitem{FS}   U. Frauenfelder and F. Schlenk,  Volume growth in the component of the Dehn–Seidel twist. {\it Geom. Funct. Anal.} {\bf 15}  (2005), 809--838.
 
  \bibitem{FM} A. Freire and R. Ma\~{n}\'{e}, On the entropy of the geodesic flow in manifolds without conjugate points.
{\it Invent. Math.} {\bf 69} (1982), no. 3, 375--392.

 \bibitem{GGM}  V.-L. Ginzburg, B.-Z. G\"{u}rel and M. Mazzucchelli, Barcode entropy of geodesic flows. To appear in Journal of the European Mathematical Society.

\bibitem{Go0} W. Gong, Infinitely many noncontractible closed magnetic geodesics on non-compact manifolds.
{\it Differential Geom. Appl.} {\bf 87} (2023), Paper No. 101977, 34 pp.

 \bibitem{Go} W. Gong, The unbounded Lagrangian spectral norm and wrapped Floer cohomology, {\it J. Geom. Phys.}  {\bf 202} (2024), Paper No. 105223, 41 pp.
 
    \bibitem{GX} W. Gong and J. Xue, Floer homology in the cotangent bundle of a closed Finsler manifold and noncontractible periodic orbits. {\it Nonlinearity} {\bf 33} (2020), no. 12, 6297--6348.
    
    

 
 \bibitem{Ho} H. Hofer, On the topological properties of symplectic maps. {\it Proc. Roy. Soc. Edinburgh Sect. A } {\bf 115} (1990),  25--38.
	
\bibitem{HZ} H. Hofer, and E. Zehnder, \textit{Symplectic invariants and Hamiltonian dynamics}. Birkh\"{a}user Advanced Texts: Basler Lehrb\"{a}cher. Birkh\"{a}user Verlag, Basel, 1994. xiv+341 pp.


\bibitem{LNV} D. Le Peutrec, F. Nier and C. Viterbo, Bar codes of persistent cohomology and Arrhenius law for  $p$-forms. {\it  Astérisque} {\bf 450} (2024), viii+194 pp.

\bibitem{LSV} 
F. Le Roux, S. Seyfaddini and C. Viterbo, Barcodes and area-preserving homeomorphisms.
 {\it  Geom. Topol.} {\bf  25} (2021), 2713--2825.
 
 \bibitem{Ka} A. Katok. Entropy and closed geodesics. {\it Ergod. Theory Dyn. Syst.} {\bf 2} (1982) 339--365.
 

 \bibitem{KT} A. Katok and J.P. Thouvenot,  Slow entropy type invariants and smooth realization of commuting measure preserving
transformations. {\it Ann. Inst. H. Poincar\'{e} Probab. Stat.} {\bf 33}  (1997), 323--338.
 
 \bibitem{KS} A. Kislev  and E. Shelukhin,  Bounds on spectral norms and barcodes. {\it Geom. Topol.} {\bf 25} (2021), 3257--3350.
 
 
 \bibitem{Ku} A.G. Kushnirenko, Metric invariants of entropy type, {\it Uspehi Mat. Nauk.} {\bf  22} (1967), 57--65.
 
\bibitem{Ma} G. A. Margulis. Applications of ergodic theory to the investigation of manifolds of negative
curvature. {\it Functional Analysis and Its Applications. } {\bf 3} (1969) 335–336.

\bibitem{Ma2} A. Manning,
Topological entropy for geodesic flows.
{\it Ann. of Math. (2)} {\bf 110} (1979), no. 3, 567--573.
 
\bibitem{Me} M. Meiwes, On the barcode entropy of Lagrangian submanifolds.  Preprint arXiv:2401.07034.


\bibitem{MPRT} E. Merelli, M. Piangerelli, M. Rucco and D. Toller,  A topological approach for multivariate
time series characterization: the epileptic brain. {\it EAI Endorsed Transactions on Self-Adaptive 
Systems} {\bf16}, no. 7, 2016.


\bibitem{MNB}J.V. Michalowicz, J.M. Nichols and F. Bucholtz, \textit{Handbook of differential entropy},  CRC Press, Boca Raton, FL, 2014, xviii+226 pp.

\bibitem{Po}  L. Polterovich, {\it The Geometry of the Group of Symplectic Diffeomorphisms}, Lectures in Mathematics ETH Z\"{u}rich, Birkh\"{a}user Verlag, Basel, 2001.

\bibitem{PRSZ} L. Polterovich, D. Rosen, K. Samvelyan and J. Zhang, \textit{Topological Persistence in Geometry and Analysis}.
University Lecture Series, vol. 74. American Mathematical Society, Providence, RI (2020). 

\bibitem{PS} L. Polterovich and E. Shelukhin, Autonomous Hamiltonian flows, Hofer's geometry and persistence modules.  {\it Selecta Math. (N.S.)} {\bf 22} (2016), 227--296.

\bibitem{PSS} L. Polterovich, E. Shelukhin and V. Stojisavljevi\'{c}, Persistence modules with operators in Morse and Floer theory. {\it Mosc. Math. J.} {\bf  17} (2017),  757--786.

\bibitem{RCMP}  M. Rucco, F. Castiglione, E. Merelli and M. Pettini, Characterisation of the idiotypic immune
network through persistent entropy.  in \textit{Proceedings of ECCS 2014}, pp. 117--128, 2014.

\bibitem{RGJA} M. Rucco, R. Gonzalez-Diaz, M. Jimenez, N. Atienza, and et al. A new topological entropybased approach for measuring similarities among piecewise linear functions. {\it Signal Process} {\bf 134} ( 2017), 130--138.

\bibitem{Sh} C. Shannon,  A mathematical theory of communication. {\it Bell System Technical Journal,} {\bf
 27} (1948), no. 3, pp. 379--423.

\bibitem{She} E. Shelukhin, On the Hofer-Zehnder conjecture. {\it Ann. of Math. } {\bf195} (2022),  775--839.

\bibitem{Se} P. Seidel. {\it A biased view of symplectic cohomology.} Current developments in mathematics,
2006, 211253, Int. Press, Somerville, MA (2008).

\bibitem{Us1} M. Usher, Duality in filtered Floer-Novikov complexes. {\it J. Topol. Anal.}  {\bf 2} (2010), 233--258.

\bibitem{Us2} M. Usher,  Hofer's metrics and boundary depth. {\it Ann. Sci. \'{E}c. Norm. Sup\'{e}r.} {\bf  46} (2013),  57--128 (2013).

\bibitem{Us3} M. Usher,  Boundary depth in Floer theory and its applications to Hamiltonian dynamics and coisotropic submanifolds. {\it Israel J . Math.} {\bf 184}  (2011), 1--57.

\bibitem{UZ} M. Usher and J. Zhang,  Persistent homology and Floer-Novikov theory. {\it Geom. Topol.} {\bf 20} (2016), 3333--3430.

\bibitem{Vi} C. Viterbo, Functors and computations in Floer homology with applications, I,
{\it Geom. Funct. Anal.} {\bf  9} (1999), 985--1033.

\bibitem{Wa} P. Walters,  {\it An introduction to ergodic theory}. Grad. Texts in Math., 79 Springer-Verlag, New York-Berlin, 1982, ix+250 pp.

\bibitem{WSBGL} X. Wang, F. Sohel, M. Bennamoun, Y. Guo, and H. Lei, Scale space clustering evolution for
salient region detection on 3d deformable shapes. {\it Pattern Recognit.} {\bf 71} (2017), 414--427.

\bibitem{Wan} Z.Q. Wang, Equivariant Morse theory for isolated critical orbits and its applications to nonlinear problems. {\it Lecture Notes in Math.}, vol.1306, Springer, 1988, pp.202–221.

\bibitem{ZC} A. Zomorodian and G. Carlsson, Computing persistent homology. {\it Discrete
Comput. Geom.} {\bf 33} (2005), 249--274.

\end{thebibliography}
\end{document}